\documentclass{amsproc}
\usepackage{amssymb}
\usepackage{amsmath}
\usepackage{amsfonts}
\usepackage{graphicx,caption}
\usepackage{setspace}
\usepackage{amsthm}
\usepackage{fullpage}
\usepackage[autostyle=try]{csquotes}
\usepackage{pb-diagram}
\usepackage{float}

\newcommand{\plane}{\mathbb{C}}

\newcommand{\disk}{{\mathbb D}}

\newcommand{\ucirc}{ {\mathbb S}}

\newcommand{\0}{\emptyset}
\newcommand{\lam}{{\mathcal L}}

\newcommand{\cl}{\overline}

\newtheorem{thm}{Theorem}[section]
\newtheorem{lem}[thm]{Lemma}
\newtheorem{cor}[thm]{Corollary}
\newtheorem{prop}[thm]{Proposition}
\theoremstyle{definition}

\newtheorem{ques}{Question}

\newtheorem{defn}[thm]{Definition}

\theoremstyle{remark}
\newtheorem{rem}[thm]{Remark}

 \newcommand{\ol}{\overline}

\newcommand\sphere{\mathbb{C}_\infty}

\usepackage[all,cmtip]{xy}

\title{Fixed Flowers}
\author{Md. Abdul Aziz}
\email[Md. Abdul Aziz]{azizm@uab.edu}
\author{Brittany E. Burdette}
\email[Brittany E. Burdette]{beburd@uab.edu}
\author{John C. Mayer}
\email[John C. Mayer]{jcmayer@uab.edu}
\date{November 2022}

\begin{document}

\begin{abstract}
    Laminations are a combinatorial and topological way to study Julia sets. Laminations give information about the structure of parameter space of degree $d$ polynomials with connected Julia sets. We first study fixed point portraits in laminations and their respective global count. Then, we investigate the correspondence between locally unicritical laminations and locally maximally critical laminations with rotational polygons. The global correspondence has been shown in \cite{Burdette:2022}. 
\end{abstract}

\maketitle



    

\section{Introduction}

\subsection{Motivation}

What are ``rotational sets" of laminations of the closed unit disk $\ol\disk$ under the action of the {\em angle $d$-tupling map}, $\sigma_d$ on the unit circle $\ucirc$ and why try to understand their relationship to fixed objects in a lamination? Laminations are a topological and combinatorial model of the connected Julia sets of polynomials considered as functions of the complex plane $\plane$.  Such models are used both to understand specific types of polynomials and their Julia sets and to study the parameter spaces of polynomials. For example, the well-known Mandelbrot set is the parameter space of quadratic polynomials (degree $d=2$) with connected Julia set. The so-called hyperbolic components of that parameter space are of interest, including how they are connected to each other, how they are arranged in the Mandelbrot set, and how many components there are that are associated with attractive orbits (of the associated polynomials) of a given period, fixed points of a given rotation number, and the like.  These terms are all defined below.  Our research is concerned with rotation sets in laminations of polynomials of higher degree ($d>2$) about which much less is currently understood.

Laminations are composed of {\em leaves} (chords of the closed unit disk $\ol\disk$) which form a closed collection of non-crossing line segments that are forward and backward invariant under a natural extension of the  map $\sigma_d$ (defined below) on the unit circle $\ucirc$. A {\em rotational set} (defined below) is a set of points of $\ucirc$ carried forward by $\sigma_d$ preserving their circular order. Leaves connecting points of a rotational set in circular order form polygons in the lamination. There is a correspondence between rotational polygons in laminations and fixed points of polynomials that have a non-zero infinitesimal rotation number (determined by the argument of the derivative of the polynomial at the fixed point). Such polygons are in correspondence to a fundamental class of hyperbolic components of the parameter space of degree $d$ polynomials with connected Julia set. Fixed points of polynomials corresponding to such rotational sets were studied in the polynomial setting by Goldberg in \cite{Goldberg1} and by Goldberg and Milnor in \cite{Goldberg2}. By studying them in the laminational setting, we do not first assume that there is a corresponding polynomial.


The connection between the Julia set of a polynomial and its corresponding lamination is made by B\"{o}ttcher's version of the Riemann Mapping Theorem.  If the Julia set of the polynomial $P(z)$ of degree $d$ is connected, then the unbounded component $B_\infty$ of the complement of the Julia set is topologically a disk  about the point $\infty$.  Let $\disk_\infty$ denote the complement of the closed unit disk in the compactified plane, the Riemann sphere $\sphere$.  By 
B\"{o}ttcher's Theorem (shown below), there is a complex analytic homeomorphism $\phi:\disk_\infty\to B_\infty$ that conjugates the power function $z \mapsto z^d$ to the action of $P$ on $B_\infty$.  In $\disk_\infty$ there are polar coordinates, concentric circles and radial lines from the boundary $\ucirc$ of the unit disk $\disk$ out toward $\infty$. Each radial line is associated with an angle in the usual way. The map $\phi$ carries this system of coordinates onto $B_\infty$. In $B_\infty$ it is customary to call the images of the radial lines {\em external rays} encoded by their angle. We should distinguish between the Julia set, which is the boundary of $B_\infty$ and the {\em filled} Julia set, which is that boundary together with what it bounds.

In case the Julia set of $P$ is locally connected, then by a theorem of Carath\'eodory all these external rays ``land" at points of the Julia set, and every point of the Julia set is a landing point of such a ray. Often more than one external ray lands at the same point. If two rays land at the same point of the Julia set, then in the corresponding lamination, those two points on $\ucirc$ are joined by a chord of $\ol\disk$, and thus a leaf of the lamination is born. By studying laminations in the abstract, without having a particular polynomial or  Julia set in mind, we aim to reverse this process and move from a lamination to a topological Julia set to a specific polynomial and its actual Julia set. Thus, we aim to understand what is possible for laminations and use that information to constrain what is possible for locally connected Julia sets of polynomials.

\begin{figure}[H]
\includegraphics[width= 14cm]{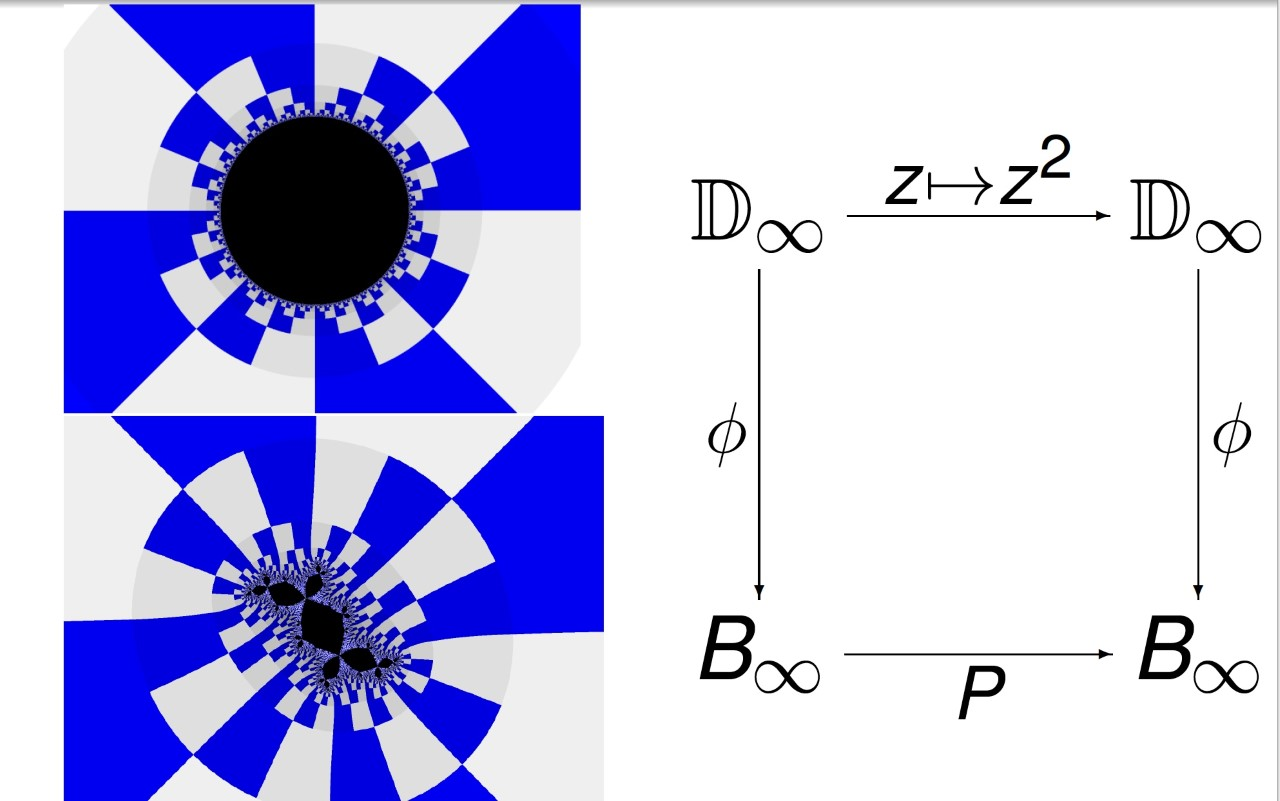}

\caption[Pinching the Circle for a Rabbit Julia Set]{On the left, we see trivial Julia set and the famous Douady Rabbit Julia set. We pinch the disk at certain angles to get the Douady Rabbit Julia set corresponding to $P(z) = z^2 + (-0.12 + 0.78i)$.  The checkerboards correspond to binary coordinates. On the right, we have the B\"{o}ttcher's Diagram mentioned above.} \label{bott}

\end{figure}



We cite \cite{Blokh:2021} for specific motivation for this paper. The authors introduce the concept of ``flower-like sets" which give motive to expand the global correspondence in \cite{Burdette:2022} to a more ``local" correspondence.

\subsection{Preliminary Definitions}

We cite the following preliminary definitions from \cite{Burdette:2022}.

\begin{defn}[$\sigma_d$ map] \label{dmap}
We use ``d-nary" coordinates on the circle $\ucirc$.
The map $\sigma_d : \ucirc \to \ucirc$ is defined to be $\sigma_d (t) = dt \pmod{1}$. For example, $\sigma_2$ has binary coordinates, $\sigma_3$ has ternary, $\sigma_4$ has quarternary, and so on.
\end{defn}

\begin{rem}
    As we are in binary in $\sigma_2$, the digits 0 and 1 are used. For example $\_001$ represents $\dfrac{1}{7}$. The digits following a ``\_" are repeated. For an arbitrary $d$-nary coordinate, a geometric series calculation can be performed to retrieve its fractional equivalent.
\end{rem}

\begin{defn}[Lamination]
A {\em lamination}, $\lam$, is a collection of chords of the closed unit disk, $\cl\disk$, which we call {\em leaves}, such that:
\begin{enumerate}
    \item Any two leaves of $\lam$ meet, if at all, in a point of the circle $\ucirc$.
    \item  $\lam^*:=\ucirc\cup\{\cup\lam\}$ is a closed subset of $\cl\disk$.
\end{enumerate}
In the case that condition (1) only is met, we call our collection of chords a {\em pre-lamination}.
\end{defn}

The following map is combinatorial; it tells where the endpoints of leaves map but not where points on the leaves map. There is an obvious continuous linear extension of this map to leaves. We introduce the extension, $\sigma^\#_d$, to the entire disk in Section \ref{section:fixed}.

\begin{defn}[Leaf Mapping]
We denote mapping a leaf, ${\ell} = \cl{ab}$, with endpoints $a$ and $b$, under the map $\sigma_d$ by $$\sigma_d (\ell) = \cl{\sigma_d(a) \sigma_d(b)}$$

\end{defn}

\begin{defn}[Critical Chords]
A chord, $\ell = \cl{ab}$, is called {\em critical} when both its endpoints map to a single point $\sigma_d(a) = \sigma_d(b)\in\ucirc$. 
\end{defn}

\begin{defn}[All Critical Polygon]
A polygon whose endpoints all map to a single point is called an {\em all critical polygon}.
\end{defn}

\begin{defn}[Sibling Leaves]
Let $\ell_1\in\lam$ be a  leaf  and suppose
$\sigma_d(\ell_1)=\ell'$, for some non-degenerate leaf $\ell'\in\lam$.
A  leaf
$\ell_2\in\lam$, disjoint from $\ell_1$, is called a {\em sibling} of
$\ell_1$ provided $\sigma_d(\ell_2)=\ell'=\sigma_d(\ell_1)$.  A collection
${\mathcal S}=\{\ell_1, \ell_2, \dots, \ell_d\}\subset\lam$ is called a
{\em full sibling collection} provided that for each $i$,
$\sigma_d(\ell_i)=\ell'$ and for all $i\not=j$,
$\ell_i\cap\ell_j=\0$.
\end{defn}

\begin{defn}[Sibling Invariant Lamination]\label{def: Sibling invariant lam}
A lamination $\lam$ is said to be {\em sibling $d$-invariant} (or simply {\em invariant} if no confusion will result) provided that the following three statements hold:
\begin{enumerate}
\item (Forward Invariant) For every $\ell\in\lam$, $\sigma_d(\ell)\in\lam$.
\item (Backward Invariant) For every non-degenerate $\ell'\in\lam$,
there is a leaf $\ell\in\lam$ such that $\sigma_d(\ell)=\ell'$.
\item (Sibling Invariant) For every $\ell_1\in\lam$ with $\sigma_d(\ell_1)=\ell'$,
a non-degenerate leaf, there is a  full sibling collection $\{\ell_1, \ell_2, \dots, \ell_d\}\subset\lam$ such that  $\sigma_d(\ell_i)=\ell'$.
\end{enumerate}
\end{defn}

\begin{rem}
A sibling $d$-invariant lamination induces an equivalence relation. Two points on $\ucirc$ are equivalent if they are joined by a finite concatenation of leaves. We consider laminations for which this results in a closed equivalence relation. Thus, the sibling invariant laminations we will be considering have a fourth condition from \cite{mimbs:2013} not listed in the definition:
\begin{enumerate}
    \item[(4)] $\lam$ has finite equivalence classes, and all leaves are boundary chords of the convex hulls of equivalence classes.
\end{enumerate}

\end{rem}

\begin{defn}[Gap] A {\em gap} in a lamination, $\lam$, is the closure of a component of $\cl\disk\setminus\lam^*$.  A gap is {\em critical} iff two points in its boundary map to the same point. A gap with finitely many leaves in its boundary is usually called a {\em polygon}.  The leaves bounding a finite gap are called the {\em sides} of the polygon.  
\end{defn}

\begin{defn}[Fatou Gap]\label{def:fatougap}
A {\em Fatou gap} in a lamination is a gap whose boundary intersected with $\mathbb{S}$ contains a Cantor set. 
\end{defn}

\begin{prop}[Order Preserving \cite{mimbs:2013}]
On both finite and infinite gaps of a $d$-invariant lamination, the map $\sigma_d$ preserves circular order.
\end{prop}


\begin{defn}[Critical Portrait]\label{def: Critical portriat}
A maximal collection of critical chords for $\sigma_d$ is called a {\em critical portrait} (maximal here meaning at least $d-1$ non-crossing critical chords). When a collection of critical chords meet at endpoints, the number may exceed $d-1$ by forming all critical polygons.
\end{defn}

\begin{defn}[Critical Sectors]

A critical sector is a region $C$ in the closed unit disk that is bounded by critical chords and arcs of the circle such that the boundary of $C$ maps onto the circle with degree 1.
\end{defn}

\begin{defn}[Forward Invariant Set]
A {\em forward invariant set} consists of periodic or pre-periodic leaves or polygons that map forward preserving circular order without intersecting.
\end{defn}

\begin{defn}[Canonical Critical Portrait] \label{can crit portrait}
    Given a forward invariant set, a maximal collection of critical chords is called {\em canonical} provided all critical polygons/chords touch some vertex of the rotational polygon.
\end{defn}

\begin{rem}
    Canonical critical portraits force some sides of the rotational polygon to be major leaves; see Definition \ref{majmin}.
\end{rem}

\begin{defn}[Compatible]
A  critical portrait, $C$, is {\em compatible} with a forward invariant set, $F$, provided that if $\cup C$ meets $F$ it is only at endpoints of leaves or vertices of polygons. 
\end{defn}

\begin{defn}[Branches of the Inverse]\label{def: branches of the inverse}
Let $C$ be a critical portrait. Every critical sector $S$ in $\cl\disk$ defined by $C$ will have a function $\tau$ defined on $\partial \disk$ onto $\partial S$ that is one to one, order-preserving, and $\sigma_d \circ \tau$ is the identity on $\partial \disk$. 
\end{defn}

Combining the previous definitions we can now define a pullback scheme.

\begin{defn}[Pullback Scheme]
Let $C$ be a critical portrait and $F$ a compatible forward invariant set. The corresponding collection $PB(F,C)$ of the branches of the inverse determined by $C$ as in Definition \ref{def: branches of the inverse} gives us a {\em pullback scheme} for $F$.
\end{defn}

\begin{rem}
There can be multiple critical portraits compatible with $F$ that define the pullback scheme differently. There are some critical portraits that are compatible with $F$, but the related pullback lamination does not satisfy condition (4) of Definition \ref{def: Sibling invariant lam}. The following theorem is well known.
\end{rem}

\begin{thm} \cite{Burdette:2022}
Let $F$ be a periodic forward invariant set under $\sigma_d$, $C$ a compatible critical portrait, and the pullback scheme $PB(F,C) = \{\tau_1, \tau_2, \dots, \tau_d\}$. Let $F_0 = F,$ and $F_1 = F_0 \cup \tau_1(F_0) \cup \tau_2(F_0) \cup \dots \cup \tau_d(F_0)$. In general, for any given stage $n$ of the pull back $F_n = F_{n-1} \cup \tau_1(F_{n-1}) \cup \dots \cup \tau_d(F_{n-1})$. Let $F_\infty  = \bigcup^\infty_{n=0} F_n$, and let $ \mathcal{L} = \cl{F_\infty}$. Then, $\mathcal{L}$ is a sibling $d$-invariant lamination.
\end{thm}

\begin{defn}[Grand Orbit] \label{def: grand orbit}
Let $P$ be an arbitrary subset of a lamination (for example a point, leaf, polygon, gap, etc.) in a $d$-invariant lamination. The {\em grand orbit} $\mathcal{GO}(P)$ is defined to be 
$$\mathcal{GO}(P) = \{ \sigma_d^{-n}(\sigma_d^k(P)) | n,k \in \mathbb{Z}^+ \cup {0} \}.$$
\end{defn}

\begin{defn}[Major and Minor Leaves]\label{majmin}
For the orbit of a periodic leaf in a unicritical lamination, the leaf, $M$, closest to critical length is called the {\em major}, and the image $\sigma_d(M)$ of the major is called the {\em minor} usually denoted $m$.
\end{defn}




\begin{defn}[Hyperbolic Lamination]
A $d$-invariant lamination is said to be \emph{hyperbolic} if and only if all compatible critical chords are interior to periodic or pre-periodic Fatou gaps. 
\end{defn}
\begin{defn}[Sub-lamination]
    A \emph{sub-lamination} is a $d$-invariant lamination that is a subset of another $d$-invariant lamination.
\end{defn}

 \begin{defn}[Rotational Set]
Consider $\sigma_d$: $\mathbb{S} \rightarrow \mathbb{S}$ for a particular $d \geq 2.$

Let $P = \{ x_i | 0 < x_1 < x_2 < x_3 < ... < x_k < 1 \}$ be a finite set in consecutive order in $\mathbb{S}$. We say that $P$ is a {\em rotational set} (for $\sigma_n$) iff \begin{enumerate}
    \item $\sigma_n(P)=P$, and
    \item For $1 \leq j \leq k$, if $\sigma_d(x_j) = x_i$, set $i(j)=i$. Then for all $j$, $j-i(j) (\text{mod }k)$ is the same. 
\end{enumerate}
If (2) holds (but possibly not (1)), we say $\sigma_d$ is circular order-preserving on $P$. We call a rotational set which is a single periodic orbit, a {\em rotational orbit}. 
\end{defn} 

\begin{rem}
    A rotational set is a basic example of a forward invariant set.
\end{rem}
 
 \begin{defn}[Rotation Number]
To each rotational set we can assign a {\em rotation number}, a rational number $0 \leq \dfrac{p}{q} < 1$ in lowest terms.  
That is, let $\mathcal{O} = \{x_1 < x_2 < x_3...<x_q \}$ be a rotational periodic orbit. Suppose that $\sigma_d(x_1)=x_j$. Set $p=j-1$. The rotation number of $\mathcal{O}$ is $\dfrac{p}{q}$. Our notation is $\rho(\mathcal{O}) = \rho(x_1)=\rho(x_i) = \dfrac{p}{q}$. 
\end{defn} 
 
 
 \begin{defn}[Rotational Polygon] \label{rotational poly}
A polygon (or a leaf) $P=P_0$ is said to be {\em rotational} iff $\sigma_d|_{P_0}$ maps $P_0$ back to itself preserving circular order and $\forall i$, $\exists j \neq i, \sigma_d(p_i) = p_j$ according to the same nonzero rotation number. 

\end{defn}

\begin{rem}
    The vertices of a rotational polygon satisfy the conditions of a rotational set.
\end{rem}

\subsection{Flower-Like Sets}

The following definitions are from \cite{Blokh:2021} and give additional motivation to studying the correspondence between locally unicritical and locally maximally critical rotational polygons. 

\begin{defn}[Lap]
    A lap of $\lam$ is either a finite gap of $\lam$ or a non-degenerate leaf of $\lam$ not on the boundary of a finite gap.
\end{defn}

\begin{defn}[Flower-Like Sets]
    Suppose that $\lam$ has an infinite invariant gap $U$ or an invariant lap $G$ and an infinite gap $U$ that shares an edge with $G$ (in the latter case it follows that $U$ is periodic). Then {U} (in the former case) or the set consisting of $G$ and all periodic Fatou gaps attached to it (in the latter case) is said to be a {\em flower-like set}.  
\end{defn}

\begin{figure}[t]
    \centering
    \includegraphics[width=2.5in]{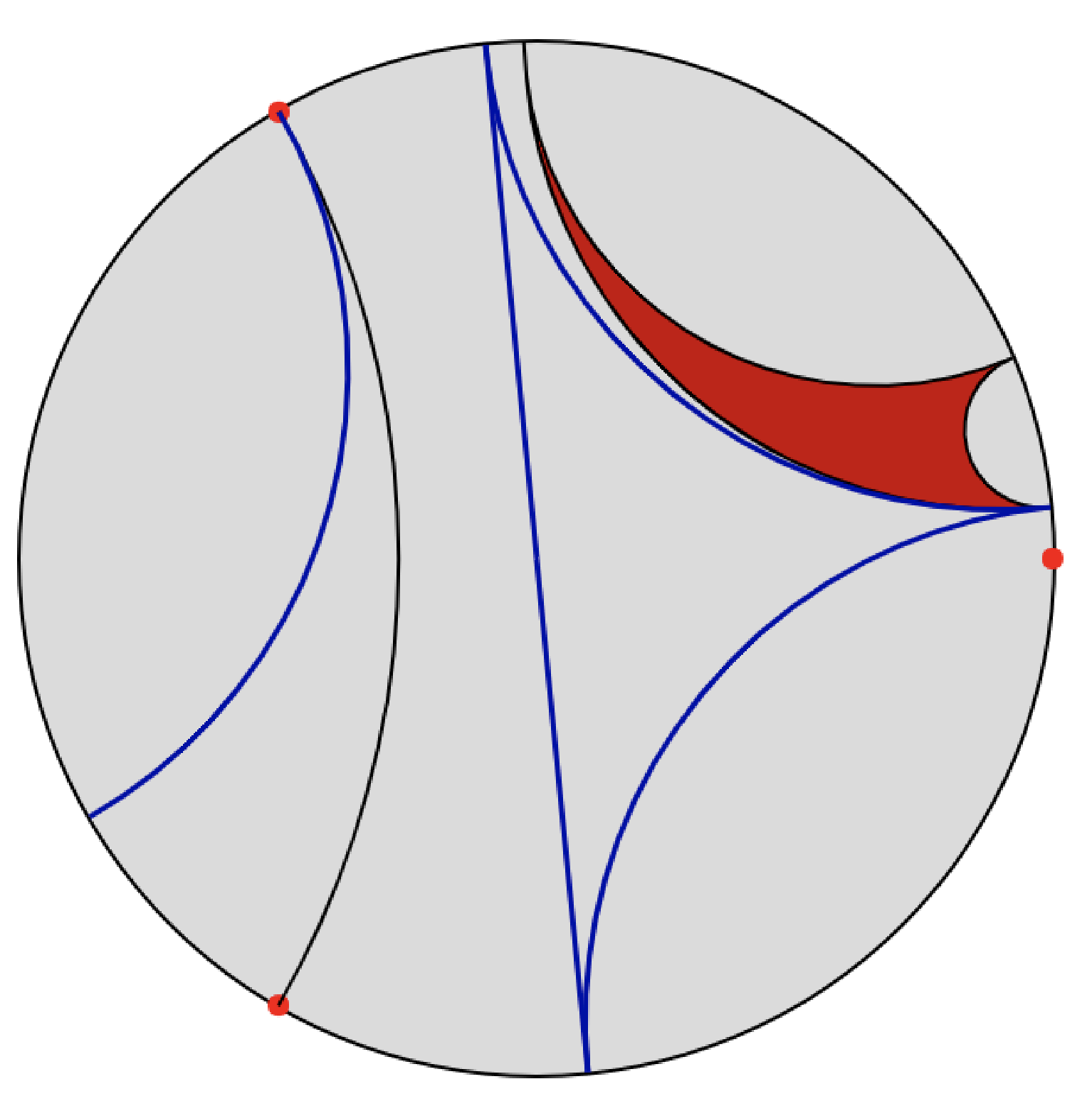}
    \includegraphics[width=2.5in]{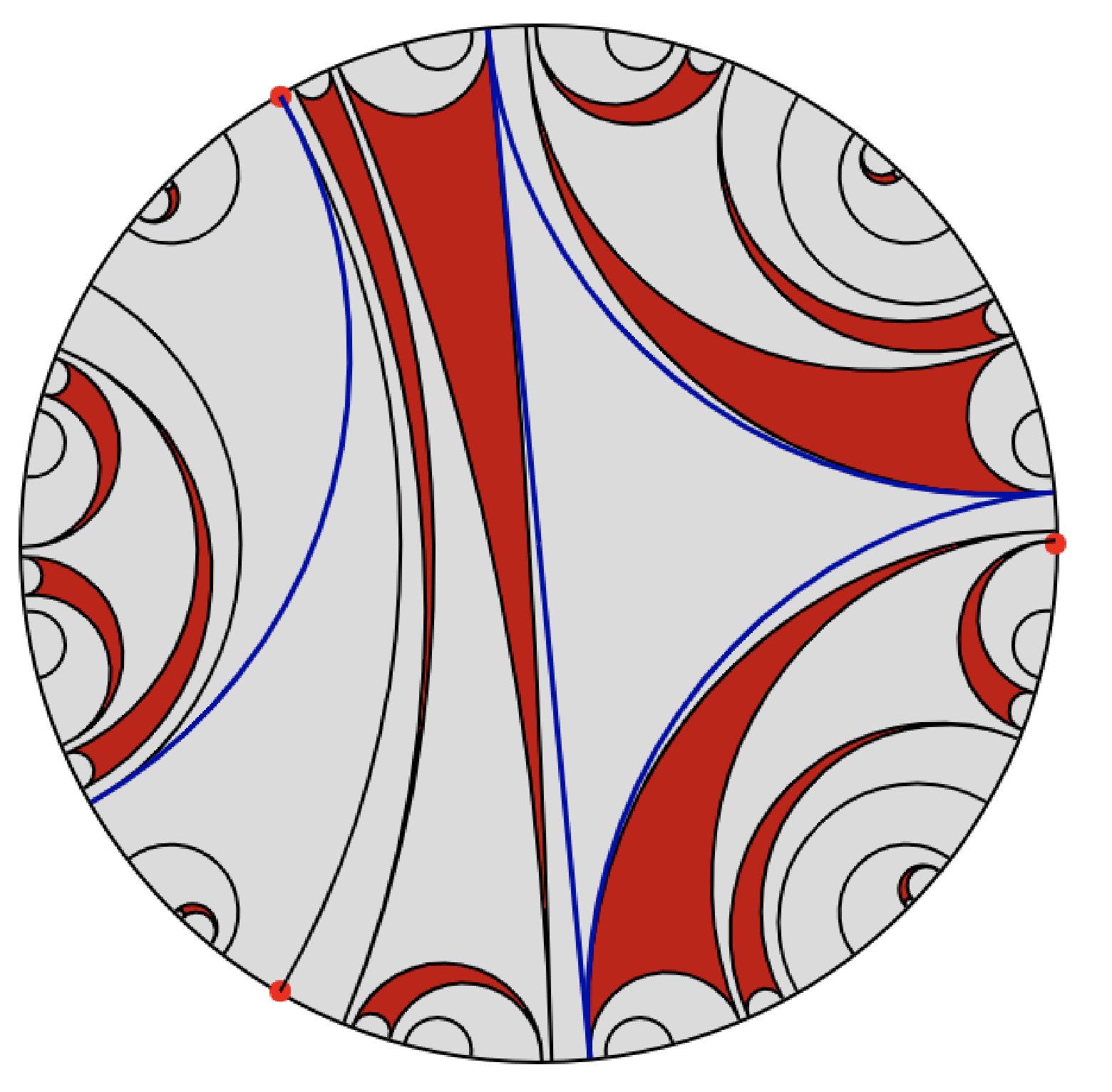}
    
    \includegraphics[width=4in]{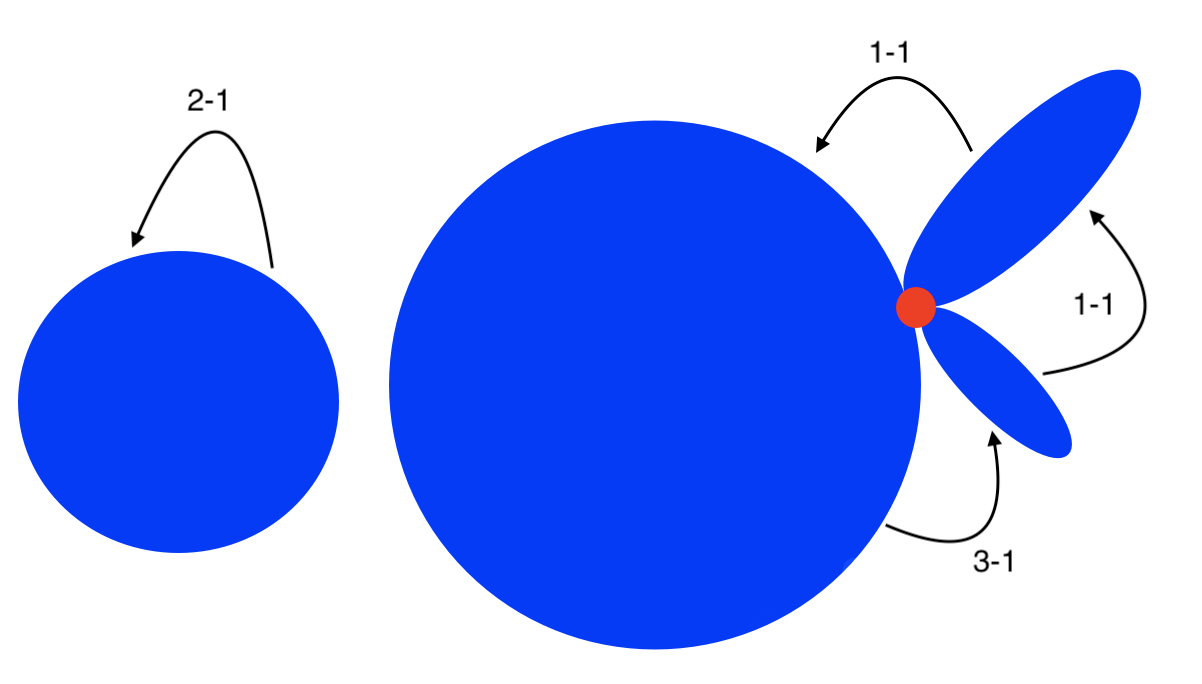}
   \caption[Fixed Rotational Polygon]{On the top left, we have a rotational polygon where only a portion of the criticality is associated with the polygon. On the top right, we have the first two pullbacks of the polygon. We can see the Fatou domains beginning to form. We also notice that the pre-images of the polygon limit to the fixed leaf. Thus, the forward orbit of the two critical Fatou domains will not meet. On the bottom, we have a sketch of the Fatou domains that are associated with criticality and how they map forward. We can see the ``flower-like" orbit  on the bottom right. The red point is the rotational polygon that was collapsed to a point in the quotient space.}
   \label{loc uni full diagram}
\end{figure}

Thus, a flower-like set is a certain set of gaps or laps.  An example which pertains to our desired correspondence is that of a rotational polygon in a fixed sector of a lamination as they will have an associated invariant gap. Fixed objects and sectors are defined in the following section. In Figure \ref{loc uni full diagram}, we show a lamination with a rotational polygon and a fixed leaf in the lamination. This fixed leaf separates the unit disk into two fixed sectors which forces our polygon to not be associated with the full degree of criticality.

\section{Fixed Point Portraits for Laminations}\label{section:fixed}

Fixed point portraits in the polynomial setting were studied by Goldberg in \cite{Goldberg1} and by Goldberg and Milnor in \cite{Goldberg2}. Our aim in starting with laminations is to provide more detail about the association of fixed points with zero rotation and fixed points with non-zero rotation, including counting all the possibilities. In this paper we restrict ourselves to the hyperbolic case.

\subsection{Fixed Point Portraits}
\begin{defn}[Fixed Point]We say a point $x$ is {\em fixed} by a map $f$, if and only if, $f(x)=x$.
\end{defn}

\begin{prop}\label{fxd}For the map $\sigma_d$ on the circle, the fixed points are $\frac{i}{d-1}$, where $0\le i<d-1$, for $i\in \mathbb{Z^{+}}$.
\end{prop}

\begin{defn}[Fixed Point Portrait (FPP)] \label{fpp} We call a partition of $\mathbb{S}$, obtained by connecting any number of fixed points with leaves in a manner that two leaves only meet at a fixed point, a {\em fixed point portrait}. We call a sector of $\mathbb{D}$ bounded by fixed leaves and arcs of $\mathbb{S}$ a {\em fixed sector} and a polygon inside $\mathbb{D}$ bounded by fixed leaves a {\em fixed polygon}.  If two fixed leaves do meet at a point, then they must be part of the boundary of a fixed polygon.
\end{defn}

\begin{rem}
    For us, a fixed point portrait  is defined to  consider only the $d-1$ fixed points that are ``externally visible"; see Figure \ref{bott}. The ``missing" fixed points are then shown to necessarily exist in the laminational setting at the end of Section~\ref{canLam} (Theorem \ref{IFP}). 
\end{rem}

Figures \ref{fig:no fixed leaves}, \ref{fig:quarter fixed leaf}, \ref{fig:no fixed leaves1}, \ref{fig:no fixed leaves2}, \ref{fig:no fixed leaves3}, and \ref{fig:no fixed leaves4} are all degree 5 fixed point portraits (up to rotation by a fixed point), a finite pullback of the lamination, and its corresponding Julia set. The fixed points on the circle in the lamination are marked with a dot. We can see the relation between laminations and Julia sets: leaves in the laminations correspond to pinch points in the Julia sets.

\begin{figure}[H]
    \centering
    \includegraphics[width = 2in]{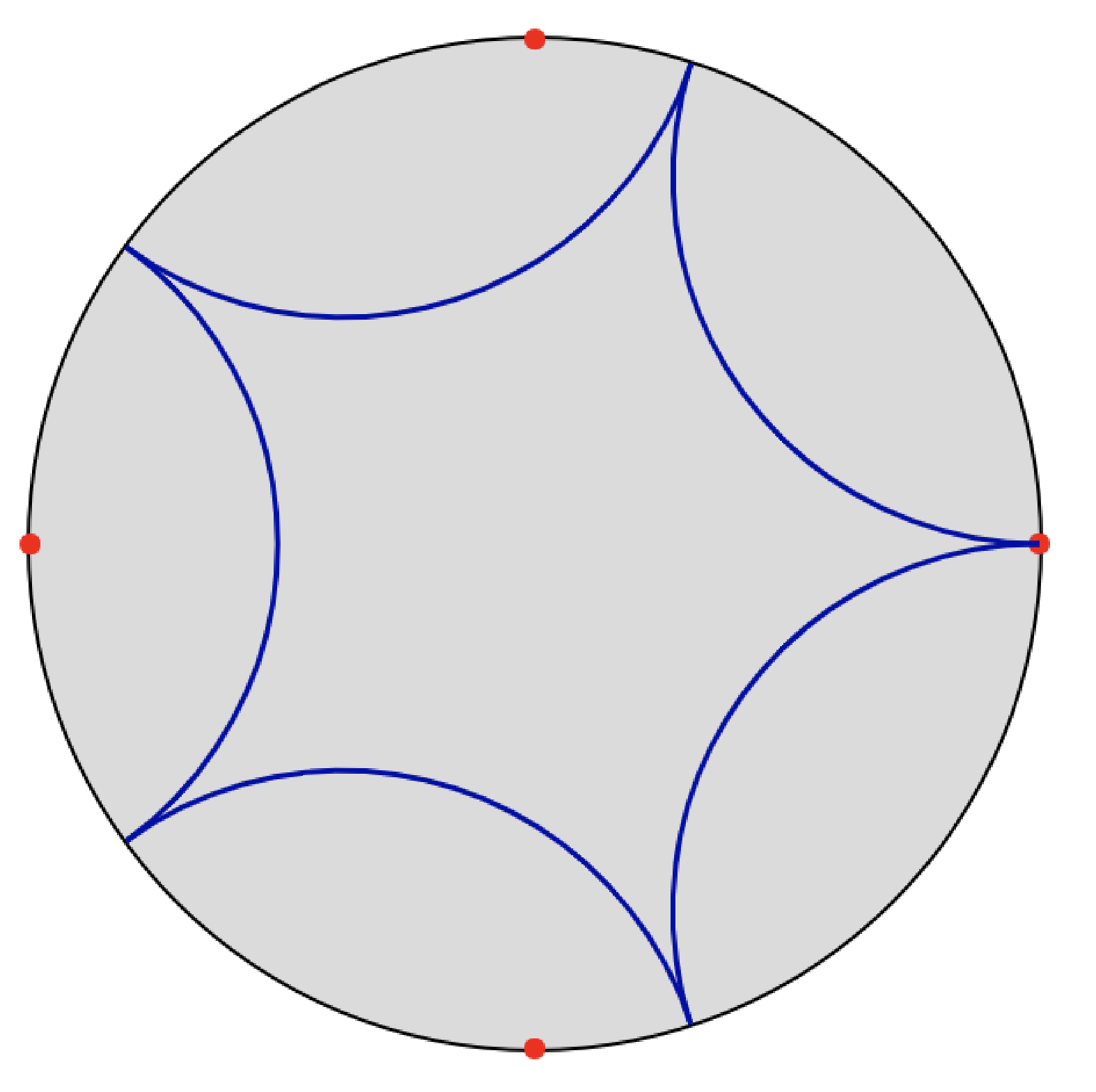}
     \hspace{.1 in}
    \includegraphics[width = 2in]{images/s5_no_fixed_leaves.png}
     \hspace{.1 in}
    \includegraphics[width = 2in]{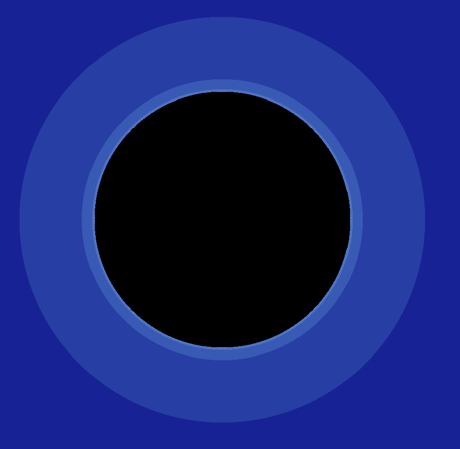}
    \caption[Simplest Fixed Point Portrait]{
   This is the simplest fixed point portrait as it has no fixed leaves only fixed points on the boundary. Here, the pullback lamination is the same as the fixed point portrait as there are no leaves to pull back. The  Julia set is just a circle.}
    \label{fig:no fixed leaves}
\end{figure}

\begin{figure}[H]
    \centering
    \includegraphics[width = 2in]{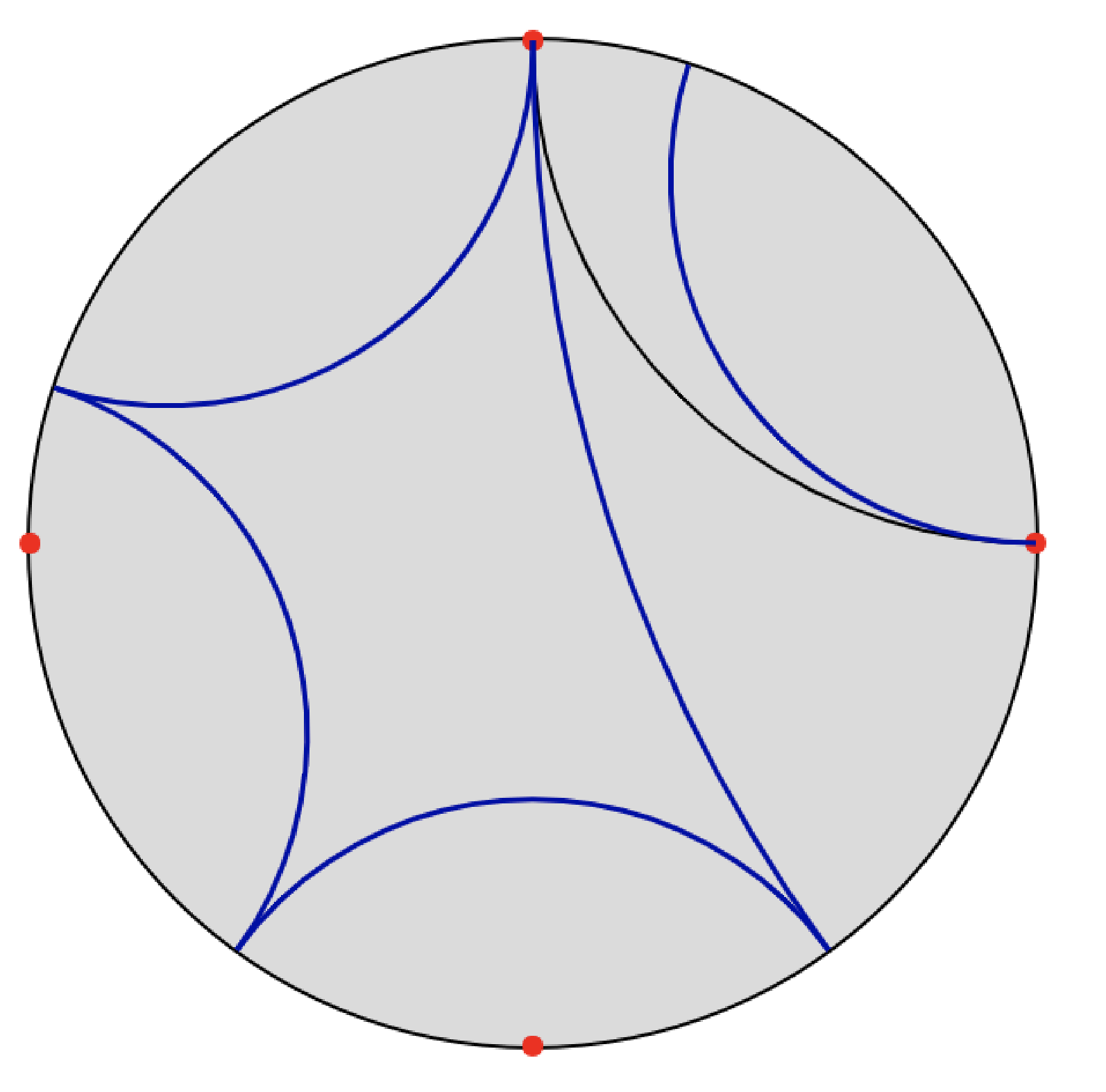}
     \hspace{.1 in}
    \includegraphics[width = 2in]{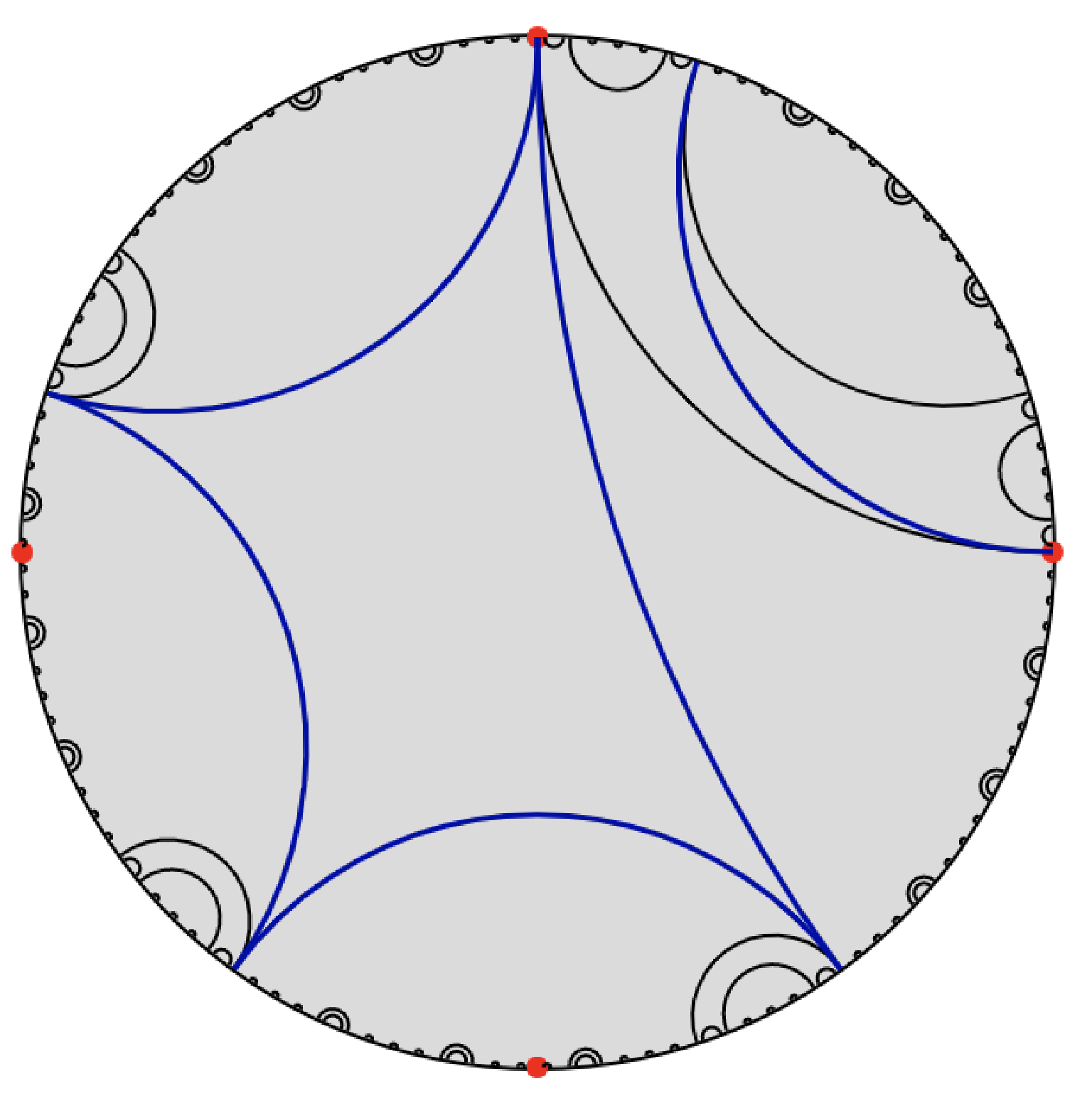}
     \hspace{.1 in}
    \includegraphics[width = 2in]{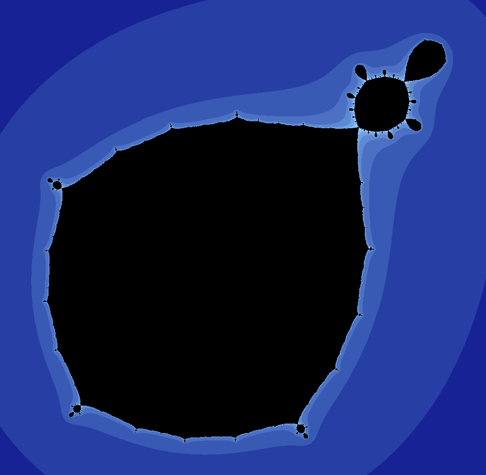}
    \caption[Simplest Fixed Point Portrait]{
   This is the fixed point portrait with a fixed leaf of length $\dfrac{1}{4}$ made by connecting \_0 and \_1 forming two fixed sectors and the associated  Julia set.}
    \label{fig:quarter fixed leaf}
\end{figure}

\begin{figure}[H]
    \centering
    \includegraphics[width = 2in]{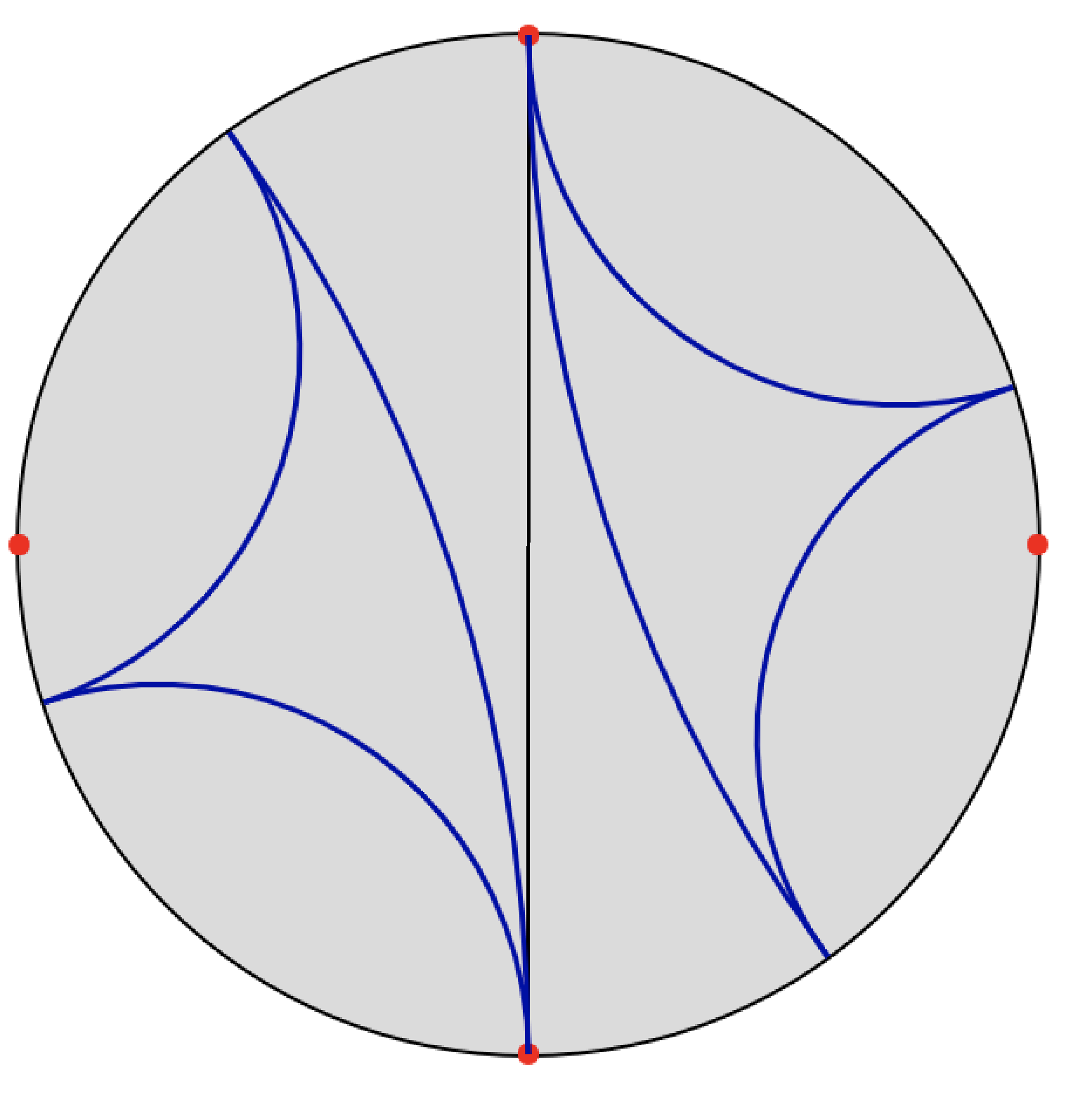}
     \hspace{.1 in}
    \includegraphics[width = 2in]{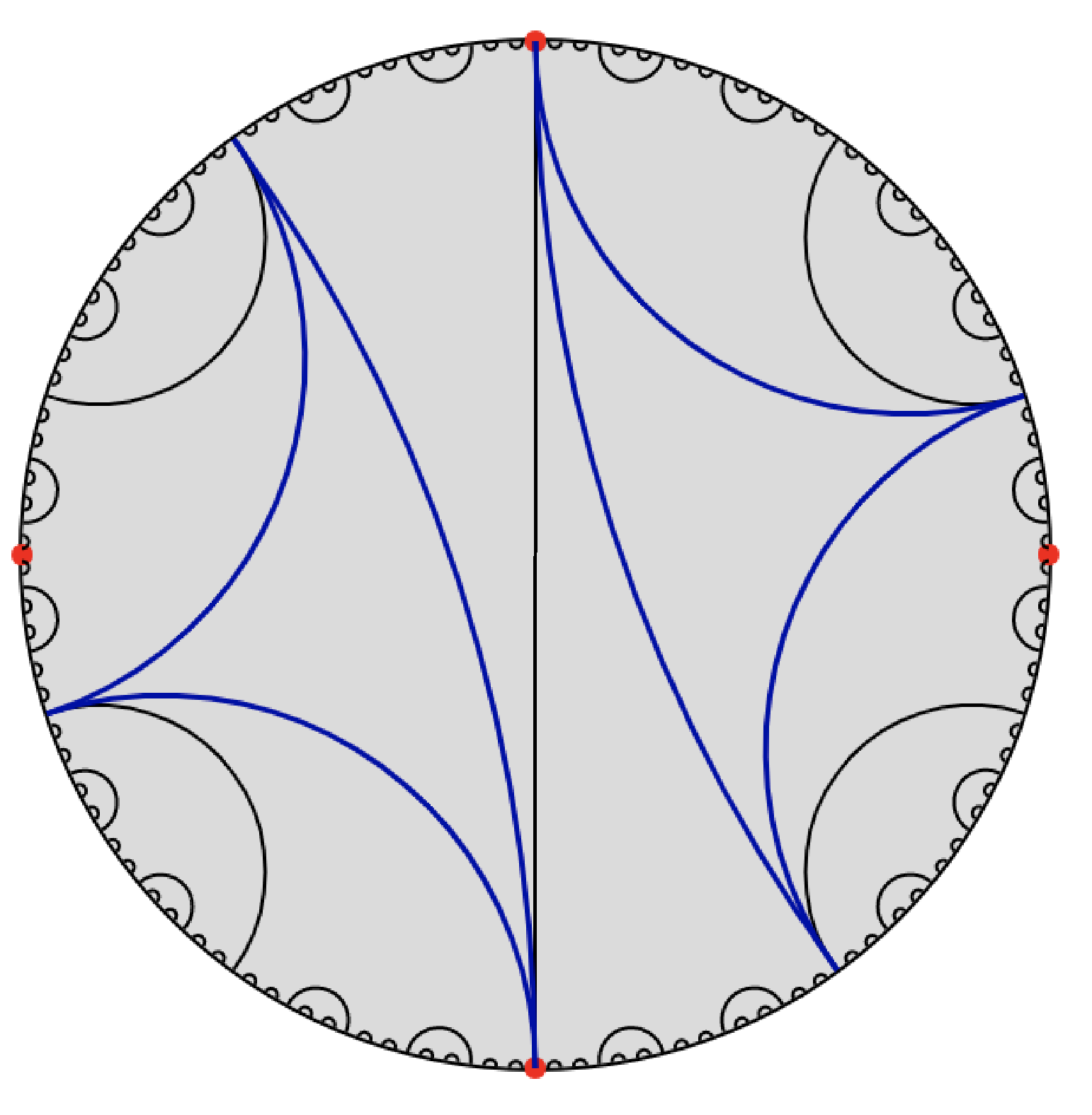}
     \hspace{.1 in}
    \includegraphics[height = 1.7in, width=2in]{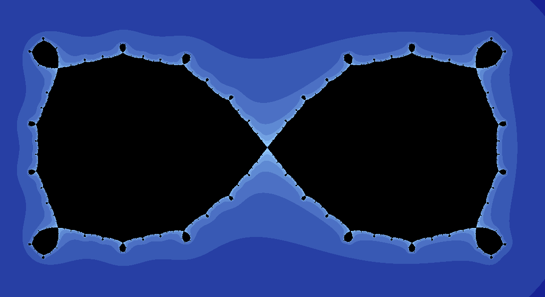}
    \caption[Simplest Fixed Point Portrait]{
   This is the fixed point portrait with a fixed diameter (length $\dfrac{1}{2}$) made by connecting \_1 and \_3 forming a rotationally symmetric FPP and the associated  Julia set with two symmetric main bulbs.}
    \label{fig:no fixed leaves1}
\end{figure}

\begin{figure}[H]
    \centering
    \includegraphics[width = 2in]{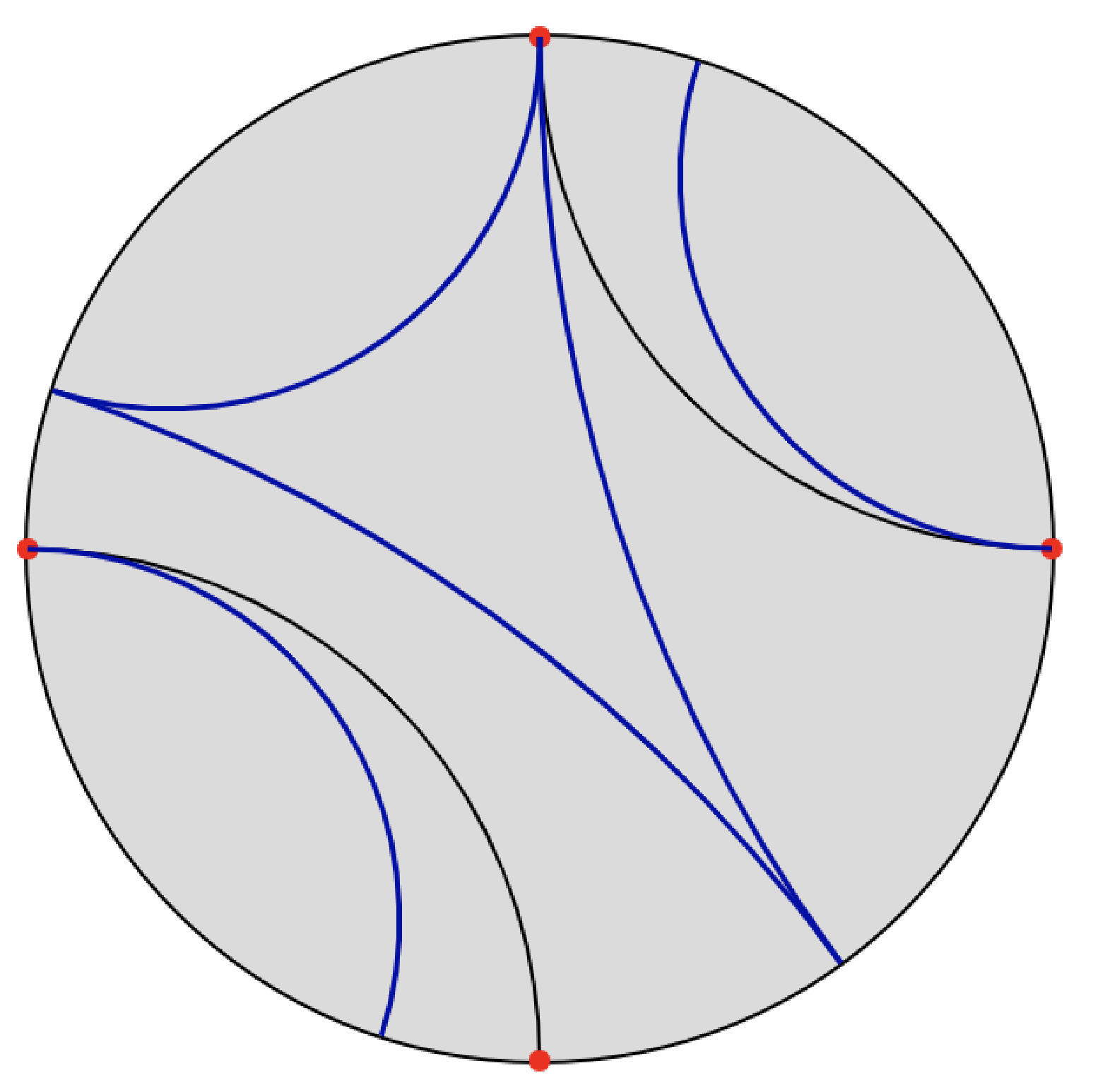}
     \hspace{.1 in}
    \includegraphics[width = 2in]{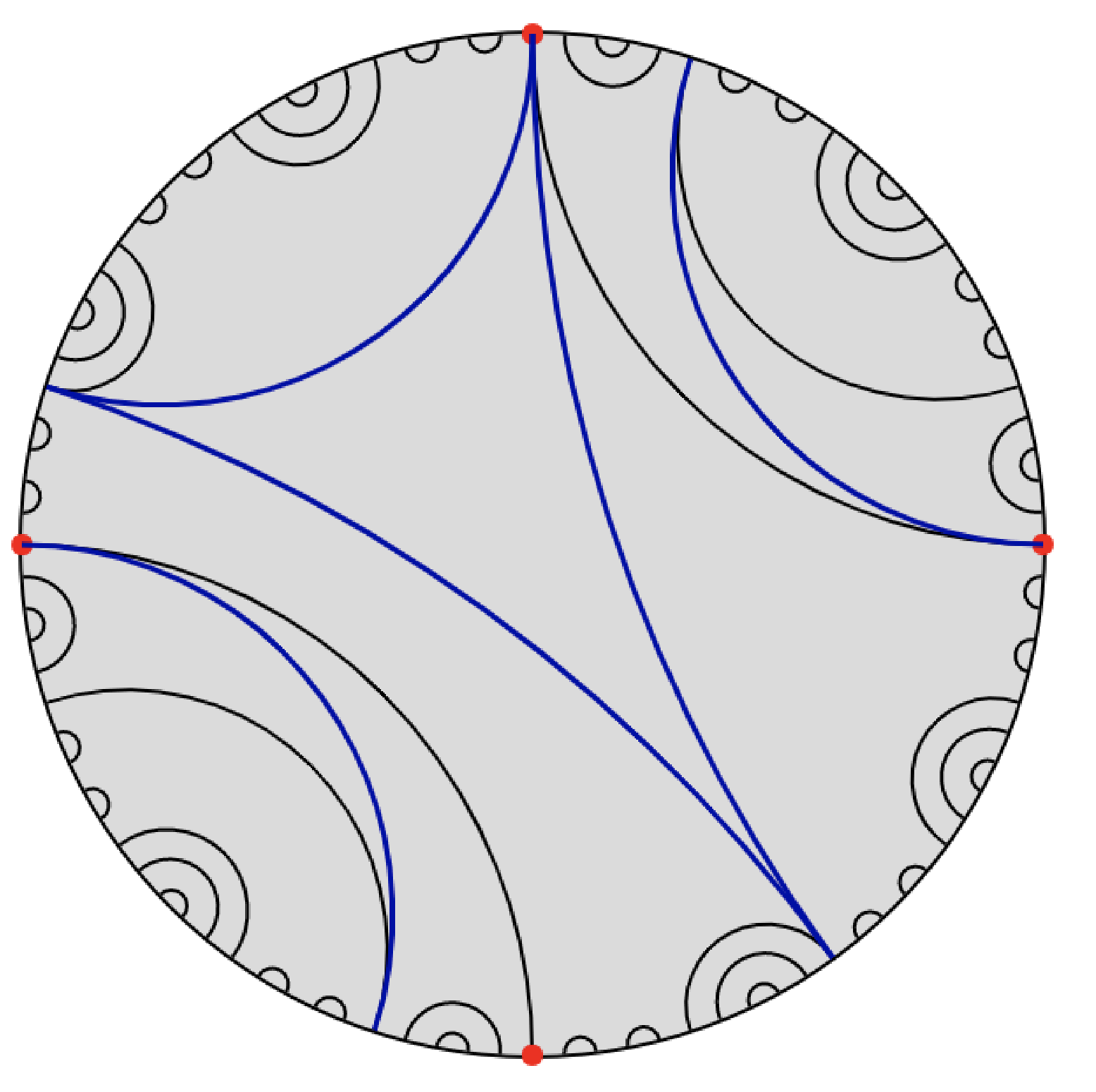}
     \hspace{.1 in}
    \includegraphics[width = 2in]{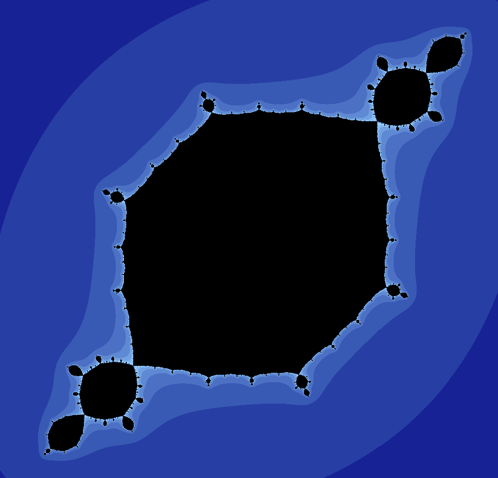}
    \caption[Simplest Fixed Point Portrait]{
   This is the fixed point portrait with two fixed leaves both of length $\dfrac{1}{4}$, one from \_0 to \_1 and another from \_2 to \_3 forming a rotationally symmetric FPP and the  associated Julia set. }
    \label{fig:no fixed leaves2}
\end{figure}

\begin{figure}[H]
    \centering
    \includegraphics[width = 2in]{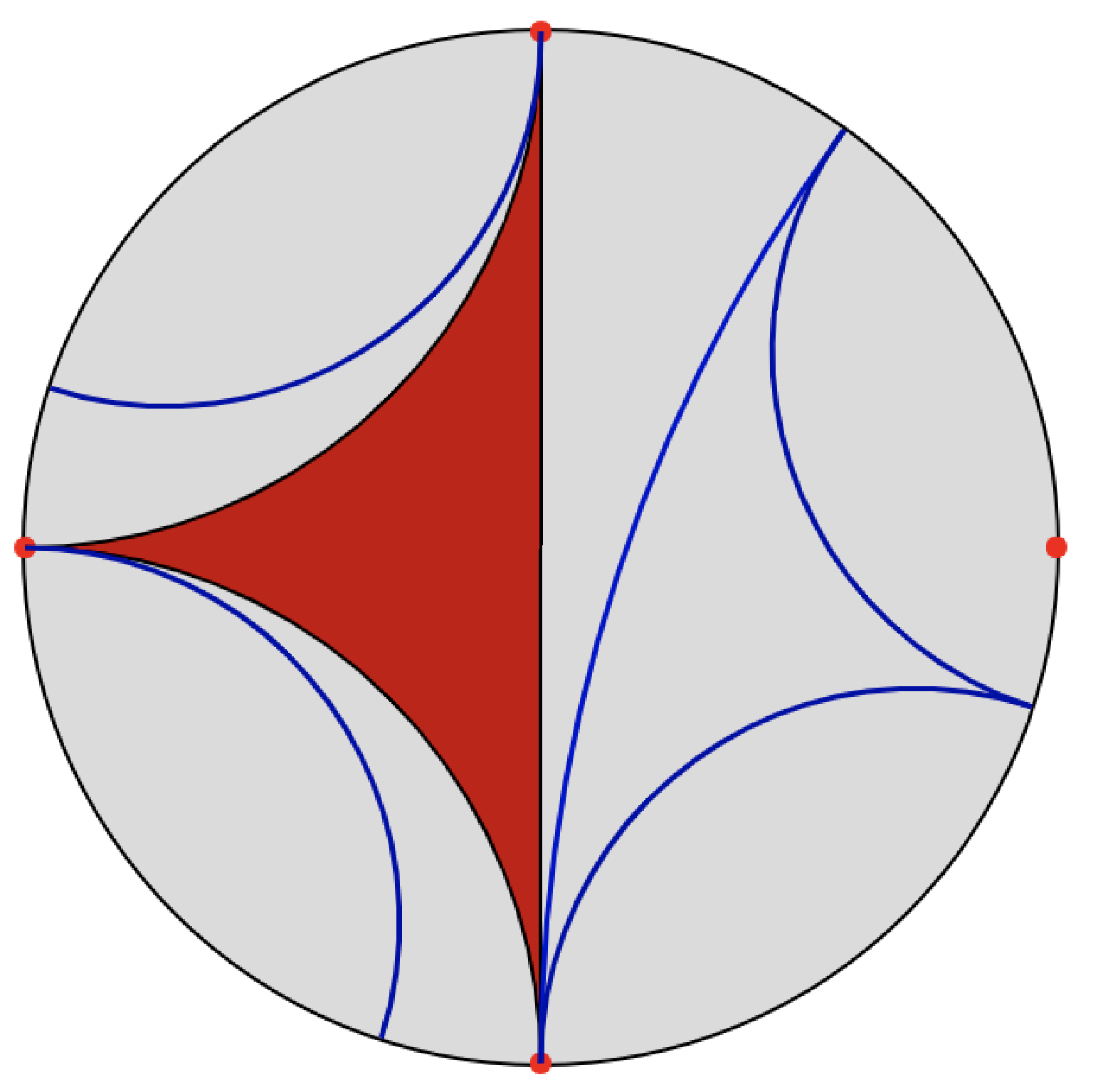}
     \hspace{.1 in}
    \includegraphics[width = 2in]{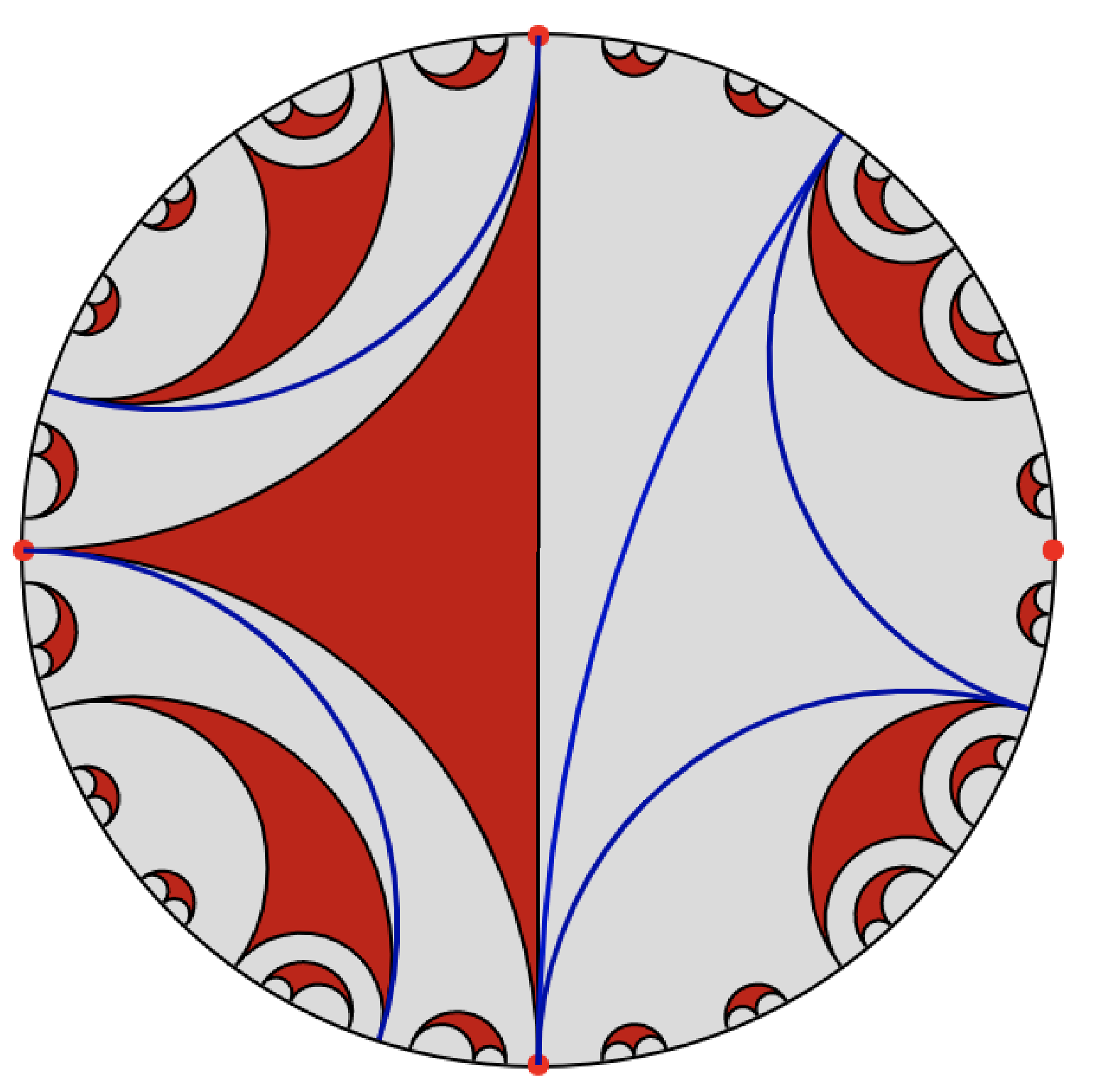}
     \hspace{.1 in}
    \includegraphics[width = 2in]{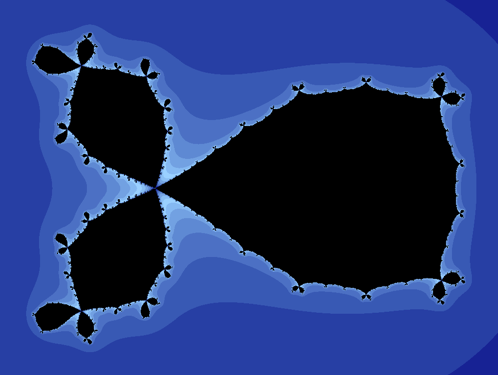}
    \caption[Simplest Fixed Point Portrait]{
   This is the fixed point portrait with a fixed triangle connecting \_1, \_2, and \_3. Here the fixed triangle has two sides of length $\dfrac{1}{4}$ and one side of length $\dfrac{1}{2}$ which has the associated Julia set having two main bulbs of the same size and one much larger main bulb. }
    \label{fig:no fixed leaves3}
\end{figure}

\begin{figure}[H]
    \centering
    \includegraphics[width = 2in]{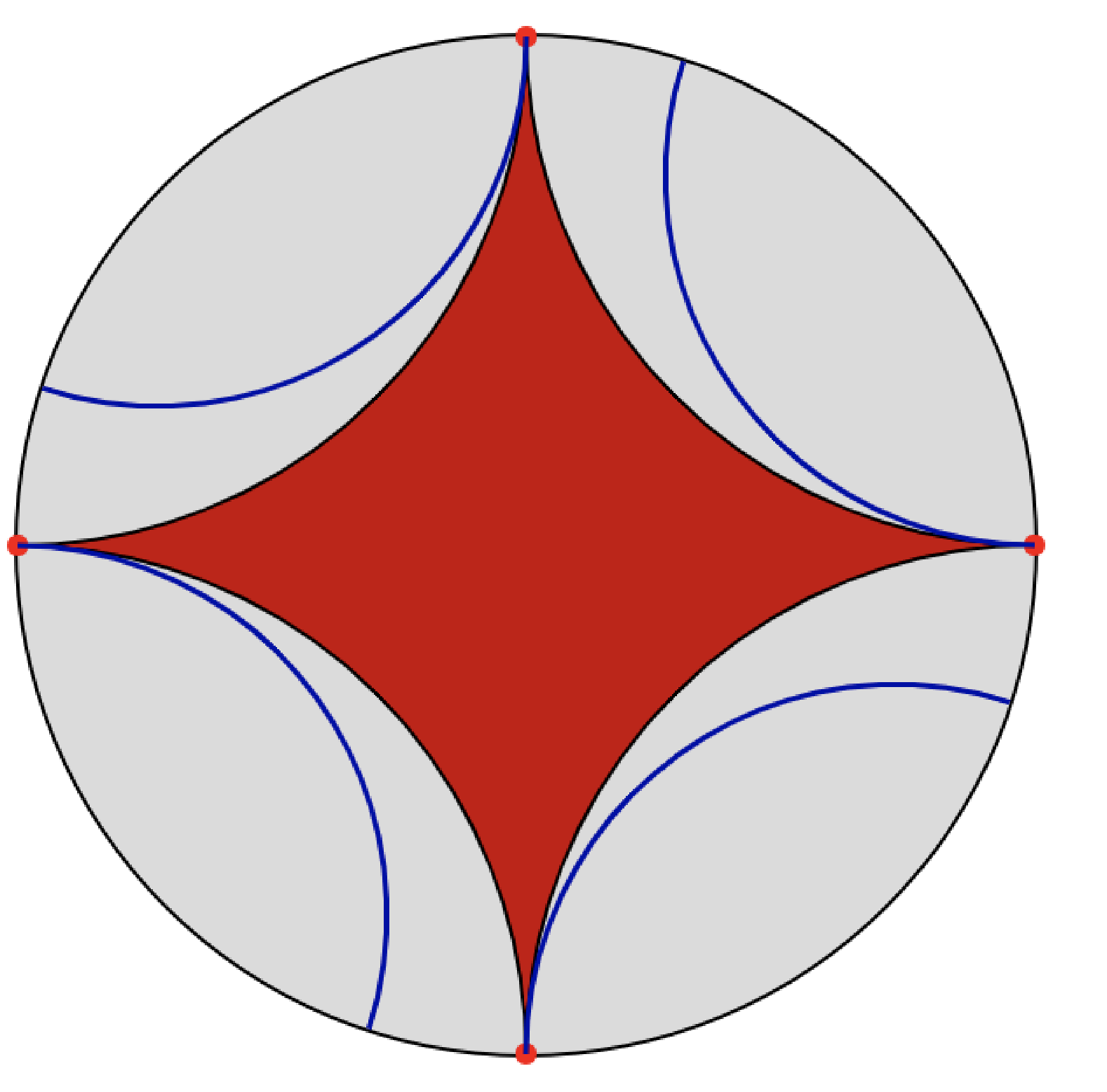}
     \hspace{.1 in}
    \includegraphics[width = 2in]{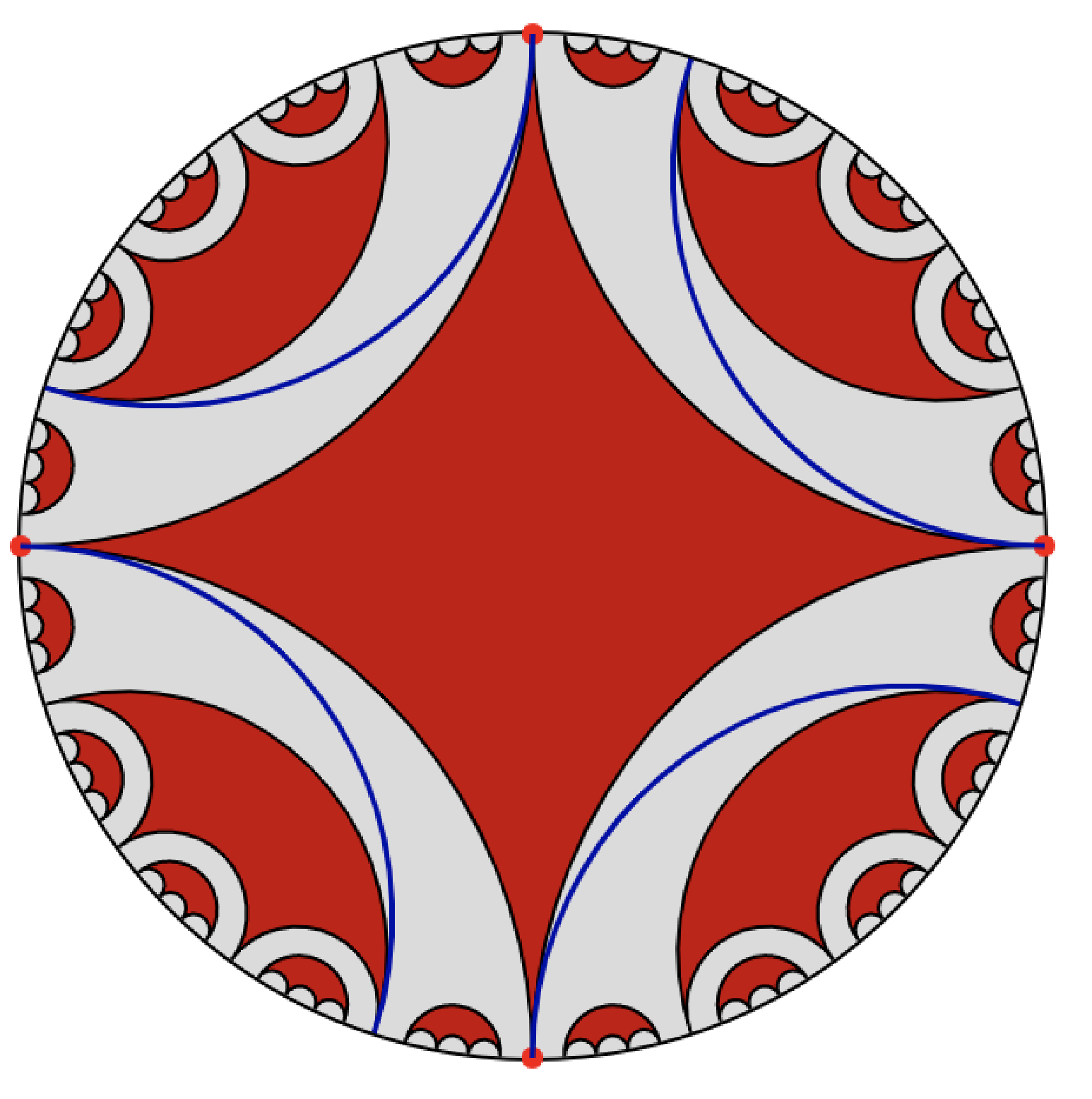}
     \hspace{.1 in}
    \includegraphics[width = 2in]{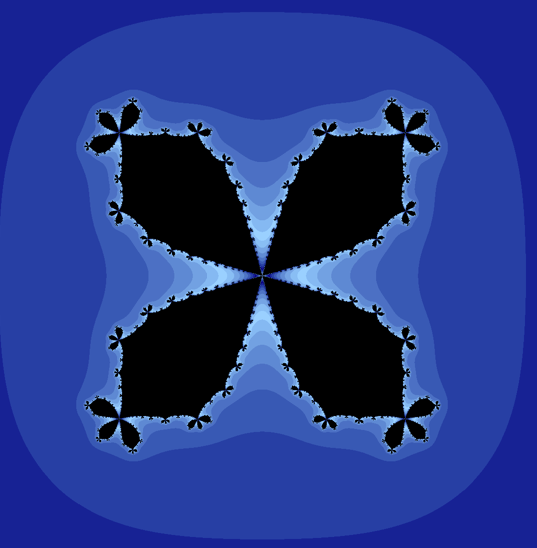}
    \caption[Simplest Fixed Point Portrait]{
   This is the maximal fixed point portrait: all fixed points are connected by leaves. The associated  Julia set has 90$^\circ$ rotational symmetry.}
    \label{fig:no fixed leaves4}
\end{figure}

\subsection{Extending $\sigma_d$ to the Disk}\label{extendsig}

In \cite{Adam:2023} and \cite{Aziz:2023}, the authors extend $\sigma_d$ on a lamination $\mathcal L$ to the closed unit disk in general, and extend to the plane as a branched covering map $\sigma_d^\#$ in the hyperbolic case in particular. Branch points (critical points) are exclusively inside Fatou gaps by the definition of hyperbolicity. The lamination induces a {\em topological Julia set} $J_{\mathcal L}$ and a branched covering map $P_{\mathcal L}$ on the plane. The first author then shows that such a branched covering map satisfies the Thurston criterion for the existence of a polynomial $P$ Thurston equivalent to  $P_{\mathcal L}$. Moreover, the lamination induced by the Julia set of $P$ is $\mathcal L$. The reader is referred to the cited dissertation and thesis for definitions and details.

\subsection{Counting Fixed Point Portraits}\label{globalcount}

Recall that we defined the map $\sigma_d$ in Definition \ref{dmap}. Also, recall that Proposition \ref{fxd} states that there are $d-1$ fixed points on the boundary of $\mathbb{S}$ for $\sigma_d$.

\begin{thm}[Counting Theorem]\label{count}The number of fixed point portraits for $\sigma_d$ are the Catalan numbers $\frac{(2n)!}{(n+1)!(n)!}$, where $n=d-1$, the number of fixed points of $\sigma_d$.
\end{thm}

\begin{proof} Let $d\ge 2$. By Proposition \ref{fxd}, the map $\sigma_d$ has $d-1$ fixed points $\{\_0,\_1,\dots,\_(n-1)\}$. We will show the number of fixed point portraits satisfies the recurrence relation of Catalan numbers.

The Catalan numbers are defined by the recurrence
$$C_0=1$$ and $$C_{n+1}=\sum_{i=0}^{n}C_iC_{n-i}$$ for $n\ge0$.

By way of induction, let $f(n)$ be the number of fixed point portraits for $\sigma_d$ using the fixed points $\{\_0,\_1,\dots,\_(n-1)\}$. We need to show $f(n)$ satisfies the above recurrence relation. We define $f(0)=1$.

Now, consider a portrait for $\sigma_{d+1}$ using fixed points $\{\_0,\_1,\dots,\_n\}$ that contains $\_n$. Also, let $\_k$ be the least fixed point in this portrait, i.e., $\_k$ and $\_n$ are connected and no fixed points of $\{\_0, \_1,\dots,\_(k-1)\}$ is connected with any fixed points of $\{\_k,\dots,\_n\}$. Therefore, $\{\_0, \_1,\dots,\_(k-1)\}$ can form portraits on their own and from the induction hypothesis, there are $f(k)$ such portraits.
Also, there are $(n-k+1)$ fixed points from $\_k$ to $\_n$. All of these fixed points lie on the arc of the circle that is bounded by the chord $\overline{\_k\_n}$. Therefore, any portraits among $\{\_k,\dots,\_n\}$ will not cross the the chord $\overline{\_k\_n}$. Furthermore, since $\_k$ and $\_n$ are already connected there are $(n-k)$ fixed points to form portraits from $\_k$ to $\_n$. Hence, by the induction hypothesis, there are $f(n-k)$ portraits.

Now, the set $\{\_0,\dots,\_n\}$ has $(n+1)$ fixed points and $\_k$ ranges ranges from $\_0$ to $\_n$. Thus we have, $$f(n+1)=\sum_{k=0}^{n}f(k)f(n-k)$$  
Therefore, by induction, we can conclude that the number of fixed point portraits of $\sigma_d$ satisfies the recurrence relation of the Catalan numbers.
\end{proof}

\section{Canonical Laminations for Fixed Point Portraits}\label{canLam}

\subsection{Canonical Critical Portraits for Fixed Point Portraits}
\begin{defn}[Degree of a fixed sector] The {\em degree of a fixed sector} is the cardinality of the maximal number of disjoint critical chords that is compatible with the region.
\end{defn}
\begin{thm}
    The degree of a fixed sector is one more than the number of arcs of the circle $\mathbb{S}$ between the adjacent fixed points in the boundary of that region. 
\end{thm}
\begin{proof}
    Each arc between two adjacent fixed points (length $\dfrac{1}{d-1}$) can contain only one critical chord (length $\dfrac{1}{d}$). See Figures \ref{fig:quarter fixed leaf} and \ref{fig:no fixed leaves2} for examples.
\end{proof}

The definition below  for fixed point portraits is similar to Definition \ref{can crit portrait}.

\begin{defn}  [Canonical Critical Portraits for Fixed Point Portrait] \label{ccp for fpp} Let $\mathcal{P}$ be a fixed point portrait. For each fixed sector in $\mathcal{P}$ a maximal all-critical polygon that touches a fixed point in the boundary of that sector is called a {\em canonical critical portrait} for $\mathcal{P}$.
\end{defn} 
In the next theorems, we will show that it does not matter which fixed point in a given fixed sector the all-critical polygon touches. This justifies calling the critical portrait canonical.

\begin{thm}\label{Compatible} There are infinitely many critical portraits compatible with the initial FPP data of leaves connecting fixed points but there are only finitely many critical portraits that are canonical, i.e., touch the fixed points.
\end{thm} 

\begin{proof}

Given any initial laminational data (including a FPP), there are an infinite number of compatible critical chords. However, if we restrict to critical portraits that touch the fixed points, there are only a finite number of ways for a chord of critical length to touch one of the finite number of fixed points.

\end{proof}
\begin{thm}\label{CPPullback} Canonical critical portraits compatible with a given fixed point portrait result in the same sibling portrait in the first pullback and thus in all subsequent pullbacks.
\end{thm}
\begin{proof} Suppose for a fixed $d$, we are given a fixed point portrait. By Proposition \ref{fxd}, the map $\sigma_d$ has $d-1$ fixed points on $\mathbb{S}$. Also, each of the $d-1$ fixed points pulls back to $d$ pre-images spaced evenly on $\mathbb{S}$, one of them the fixed point itself. Since fixed points on $\mathbb{S}$ are spaced evenly and pre-images of individual fixed points are also spaced evenly on $\mathbb{S}$, between any two consecutive fixed points there is exactly one pre-image of each of the other fixed points.

First, we consider a simple case where the fixed point portrait is just a leaf connecting any two consecutive fixed points. Then the canonical critical portrait compatible with this fixed point portrait will consist of a critical chord and an all-critical $(d-1)$-gon. The single critical chord lies in the smaller fixed sector of $\mathbb{S}$ bounded by the fixed leaf, and one of its ends touches one of the two fixed points (see Figure \ref{fig:First_pullback}). Therefore, the critical chord will connect the fixed endpoint and the pre-image of that fixed point in this fixed sector in the first pullback. The $(d-1)$-gon lies in the other fixed sector of $\mathbb{S}$ and one of its vertices touches one of the $d-1$ fixed points in this sector. Therefore, the sides of the $(d-1)$-gon will connect the consecutive pre-images of the vertex fixed point. We call these segments of $\mathbb{S}$ bounded by the critical chord and by the sides of the $(d-1)$-gon as critical intervals. Every critical interval contains pre-images of each fixed point exactly once except for the touching fixed point, which has two pre-images in every interval.

Now, in the critical interval bounded by the critical chord, we have only one choice of pre-image leaf. Thus, if we change the touching fixed point of the critical chord, it would not affect the pullback leaf and the leaf will remain the same. In the critical intervals bounded by the sides of the $(d-1)$-gon, we have exactly one pre-image of each fixed point except for the touching fixed point, which has two pre-images; we have exactly one option to pullback the fixed leaf in every critical interval. 

Moreover, since the pullback leaves are unique in every interval, shifting the critical polygon to a different touching fixed point will not affect the pullback leaves. Therefore, we can conclude that we will have the same siblings irrespective of the choice of the touching fixed point. Since we have the same sibling portrait in the first pullback and the branches of inverse are also fixed, we will have the same sibling portrait in every pullback. Therefore, the lamination is canonical.

In each fixed sector of local degree $d'$, there is an all critical $d'$-gon. The number of the pre-images of the fixed points in each critical interval of that fixed sector allow for only one way to connect pre-images of the initial FPP data. Details are left to the reader.
\end{proof}

\begin{figure}[H]
\centering
 \includegraphics[width=.4\textwidth]{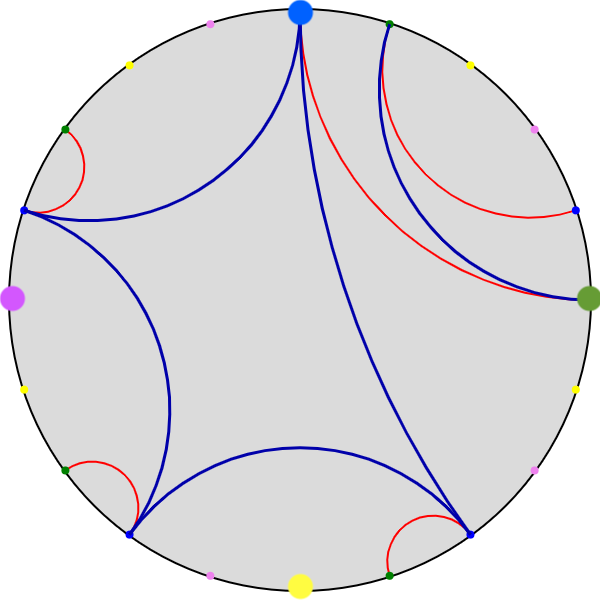} \hspace{.5in}\includegraphics[width=.4\textwidth]{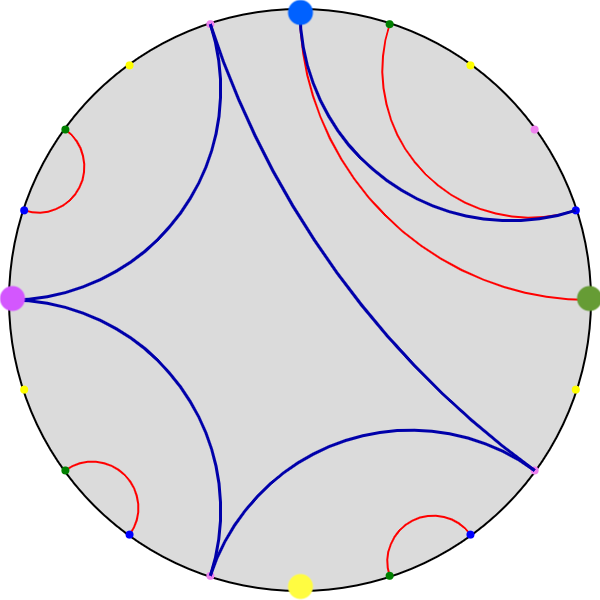}

\caption[First pullback lamination of two canonical critical portraits of $\sigma_5$]{Here we have the same fixed leaf connecting the blue and green fixed points. We choose two canonical critical portraits: one that touches at the blue fixed point and one that touches at the pink fixed point. The first pullback lamination with either critical portrait results in the same sibling portrait.}
    \label{fig:First_pullback}
\end{figure}



\subsection{Canonical Laminations for Fixed Point Portraits}

\begin{defn} [Canonical Lamination for Fixed Point Portrait] \label{clfpp} Let $\mathcal{P}$ be a fixed point portrait. Choose for each fixed sector in $\mathcal{P}$ a maximal all-critical polygon touching a fixed point in the boundary of that sector. Then the pullback lamination $\mathcal{L}$ with respect to the chosen critical portrait is the {\em canonical lamination} for $\mathcal{P}$. 
\end{defn}
It follows from Theorems \ref{Compatible} and \ref{CPPullback} that Definition \ref{clfpp} is well-defined. That is, it does not matter which canonical critical portrait is used in the pullback scheme. See Figures \ref{fig:First_pullback} and \ref{fig:Canonical}.

\begin{figure}[H]
    \centering
    \includegraphics[width = 1.8in]{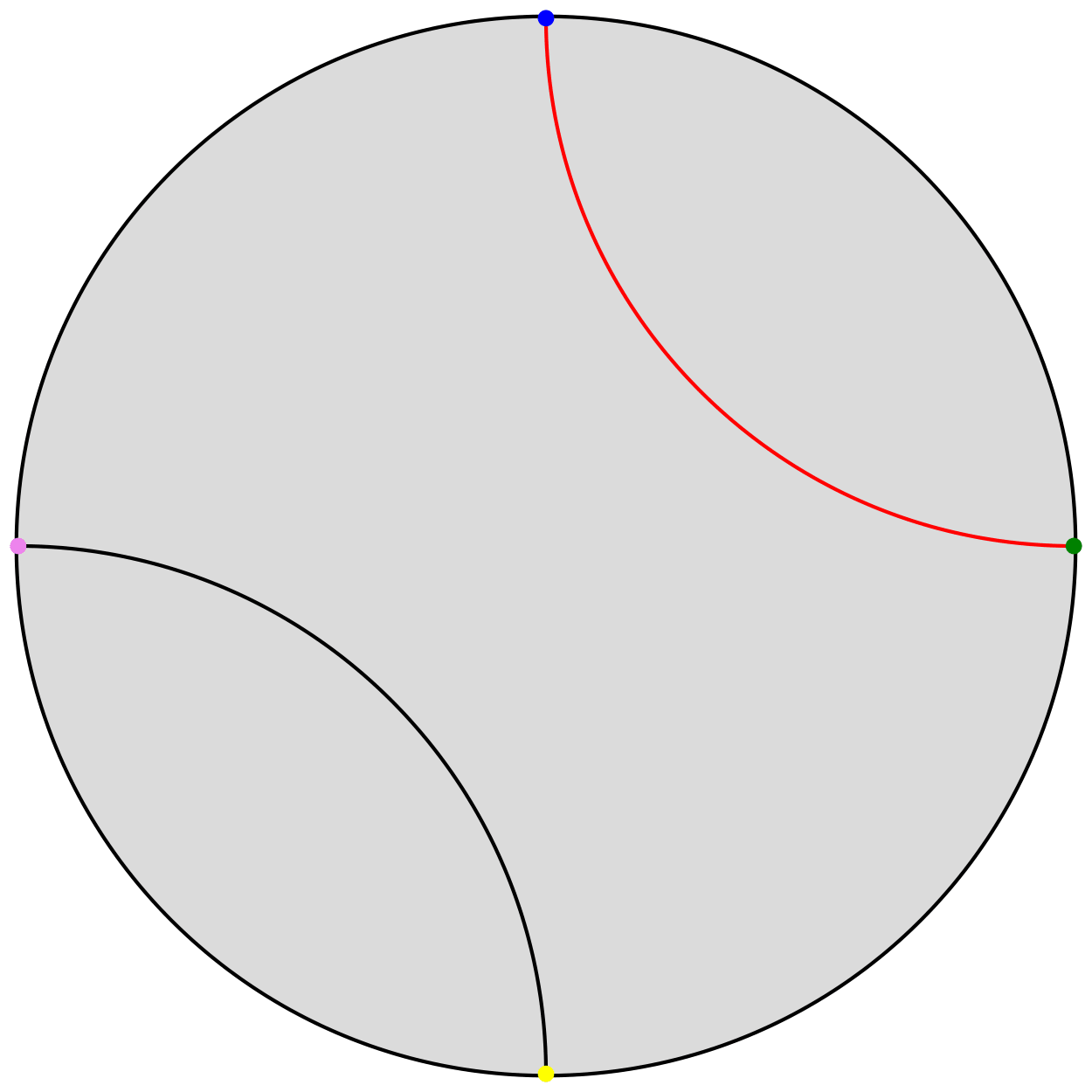}
    \includegraphics[width = 1.8in]{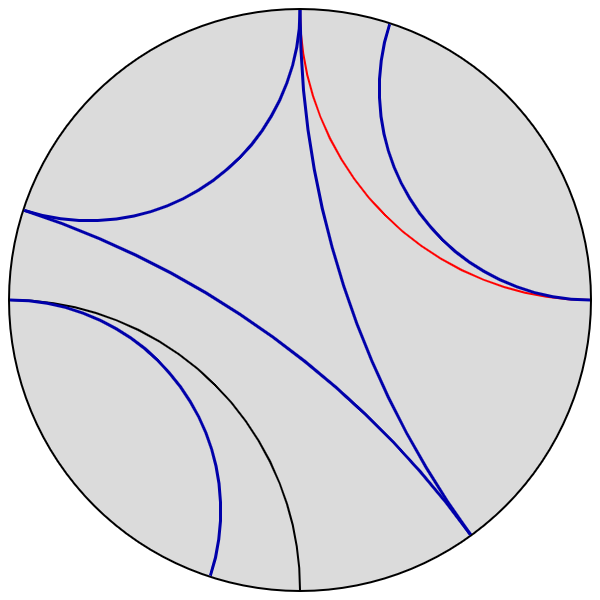}
    \includegraphics[width = 1.8in]{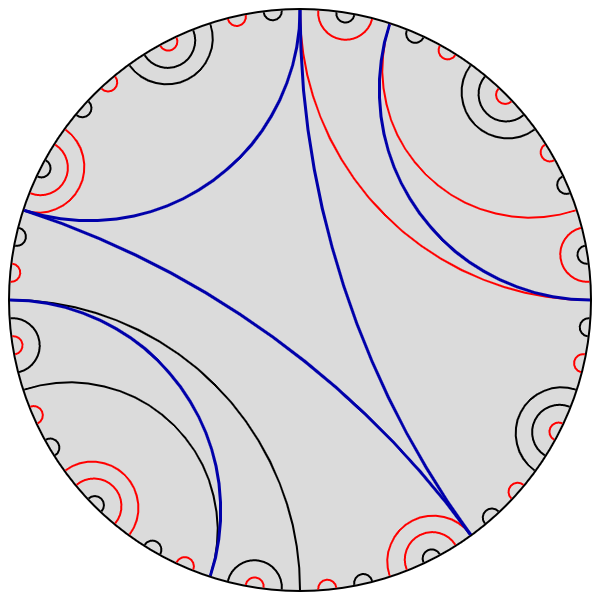}
    \caption[Canonical lamination]{
   This is a canonical lamination for $\sigma_5$ with two fixed leaves. On the left we have the two fixed leaves. In the middle, the critical portrait is added in. On the right, we have the first few pullbacks of the lamination.}
   \label{fig:Canonical}
\end{figure}

\begin{thm} [Properties of Canonical Fixed Point Lamination] \label{CLP} Let $\mathcal{P}$ be a fixed point portrait and $\mathcal{L}$ be its canonical pullback lamination. Then $\mathcal{L}$ has the following properties:
\begin{enumerate}
    \item Every pullback leaf is a pre-image of a leaf of $\mathcal{P}$.
    \item Pullback leaves converge to points of $\mathbb S$.
    \item There are no limit leaves.
    \item Every fixed sector contains a unique invariant critical Fatou gap whose center is a fixed point.
    \item The boundary leaves of the Fatou gap are pre-images of boundary leaves of that fixed sector. 
\end{enumerate}
\end{thm}
\begin{proof}
    (1) Every point on $\mathbb{S}$ for the pullback lamination $\mathcal{L}$ is a pre-image of the endpoints of the leaves of $\mathcal{P}$. So in $\mathcal{L}$, if we connect two pre-images of two endpoints of a leaf of $\mathcal{P}$ by a leaf, then this leaf in $\mathcal{L}$ maps forward to the corresponding leaf of $\mathcal{P}$. Therefore, every pullback leaf is a pre-image of a leaf of $\mathcal{P}$.\\

    \noindent (2) Let $\mathcal{L}$ be a canonical lamination of the fixed point portrait $\mathcal{P}$, i.e., $\mathcal{L}$ touches some fixed points of $\mathcal{P}$. We can see the initial pullback of a fixed leaf is trapped under a critical chord. Since the maximum length of a fixed leaf is $\dfrac{1}{2}$, the maximum length of pullback leaf is $\dfrac{1}{2d}$. Thus, $n$ pullbacks will be bounded in length by $\dfrac{1}{2d^n}$.\\

    \noindent (3) Since from (2) we have pullback leaves converge to points, there are no limit leaves.\\

    \noindent (4) Let $F$ be a fixed sector for $\mathcal{P}$. Let the degree of the sector $F$ be $n$. Then, there is a canonical all-critical $n$-gon (Definition \ref{ccp for fpp}), touching one of the fixed points on the boundary of $F$. The chords of the all-critical polygon determine the critical sectors of $F$. When we pull back, each of the critical sectors has a pre-image of every fixed point inside it. In the first pullback, we will see a pre-image of each fixed leaf in that critical sector. The boundary leaves of $F$ subtend any other fixed leaves of $\mathcal{P}$. Since the pullback of the fixed points preserves counterclockwise circular order, the pullbacks of other fixed leaves that are not boundary leaves of $F$ will be subtended by pullbacks of the boundary leaves of $F$. If we continue pulling back, we will have a gap $G$ in $F$ that meets the circle in a Cantor set. By Definition \ref{def:fatougap}, $G$ is a Fatou gap. The pullback process uniquely determines the Fatou gap.\\

    \noindent (5) From the proof of (4) we have seen that in the first pullback inside $F$, the pullback leaves that are not pullbacks of boundary leaves of $F$ are subtended by the pullbacks of the boundary leaves of $F$. Therefore, the boundary leaves of the Fatou gap are pre-images of the boundary leaves of that fixed sector.
\end{proof}

\subsection{Fixed Points in a Lamination} Every lamination of degree $d$ is supposed to correspond to a polynomial of degree $d$. Every polynomial of degree $d$ has $d$ fixed points. But on the boundary of $\mathbb{S}$ we have exactly $d-1$ fixed points. If in the lamination two fixed points are joined by a leaf, then in the corresponding Julia set this will be a multiple fixed point which will count as only one fixed point for the polynomial. For example, in Figure \ref{fig:no fixed leaves2} four peripheral fixed points in the lamination have become two multiple fixed points in the Julia set. In the canonical case, there is an invariant Fatou gap in each fixed sector as in Theorem \ref{CLP}, with an attractive fixed point in each invariant Fatou gap. In this subsection, we will show that there exists a fixed object in the interior of every fixed sector in all non-canonical hyperbolic cases as well.


\begin{figure} [h!]
    \begin{center}
 \includegraphics[width=.3\textwidth]{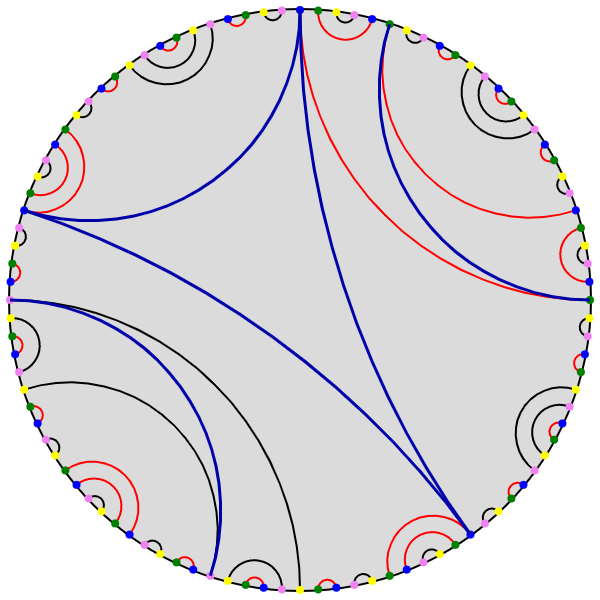}
 \includegraphics[width=.3\textwidth]{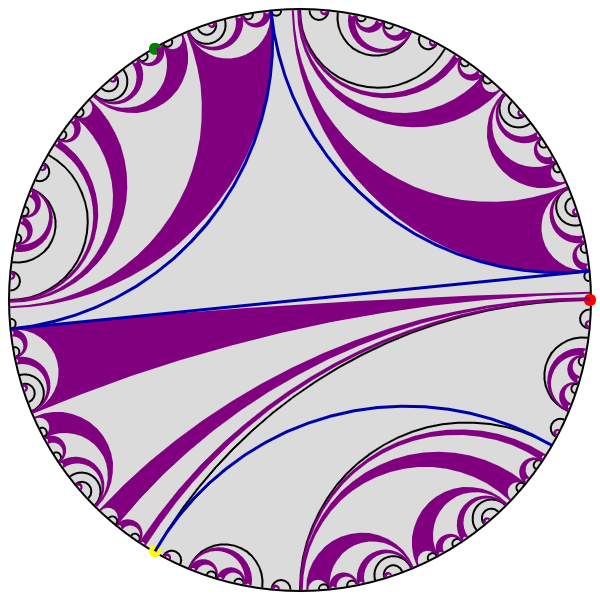}
 \includegraphics[width=.3\textwidth]{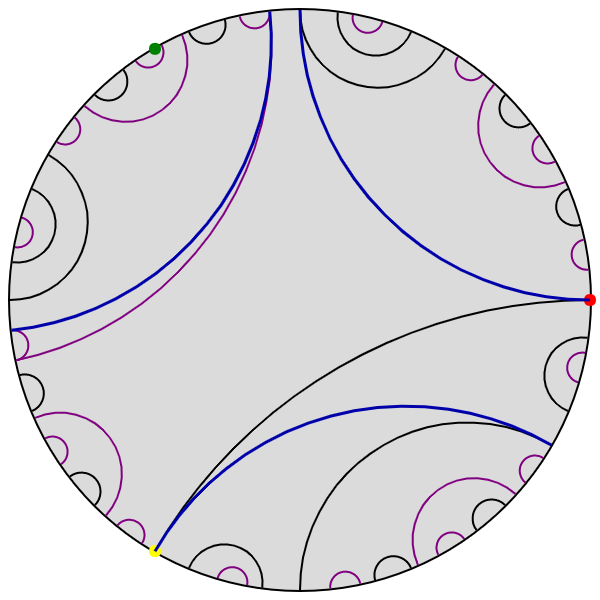}
\end{center}
    \caption[Internal Fixed Point]{Left to right cases I, II, and III of Theorem \ref{IFP}. Note that the degrees of the laminations illustrated are $5$, $4$, and $4$. On the left, we have Case I where no fixed object, say in the central degree $3$ fixed sector, is subtended by the leaves of the lamination. In the middle, we have Case II where every fixed object  in the degree $3$ fixed sector is subtended by leaves of the lamination. On the right, we have Case III  where in the degree $3$ fixed sector only some fixed objects are subtended by leaves of the lamination and some are not.}
    \label{fig:Cases}
\end{figure}

\begin{thm}\label{IFP}
Let $\mathcal{L}$ be a hyperbolic lamination containing a particular FPP $\mathcal P$. In each fixed sector of $\mathcal P$ there is a fixed object of one of the following types:
\begin{enumerate}
\item The fixed point is the center of an invariant Fatou gap.
\item The fixed object is a rotational polygon.
\end{enumerate}
\end{thm}

\begin{proof}
Let $F$ be a fixed sector of $\mathcal{P}$. The proof is done is three cases; the first and third case are of Type (1). The cases divide on whether leaves of the lamination subtend fixed points or fixed leaves of the fixed point portrait.
    
    \noindent Case I: No fixed point or fixed leaf of $F$ is subtended by the leaves of the lamination.  The left-most lamination in Figure \ref{fig:Cases} is a $\sigma_5$ example: there are two fixed leaves and three canonical invariant gaps. $F$ is the central fixed sector.\\
    
    \noindent Since no fixed point or fixed leaf is subtended by leaves of the lamination, the fixed points and leaves must be on the boundary of a gap, $G$, in $F$. Since the fixed points and leaves in the boundary of $G$ are disjoint, $G$ is a Fatou gap. Since $G$ is an invariant Fatou gap, there are points $x$ and $y$ that map to the same point \cite{Blokh:2006}. Thus, $G$ is critical. Since the lamination is hyperbolic, the criticality is interior to the Fatou gap. As the extension $\sigma_{d}^{\#}$ \cite{Adam:2023} of $\sigma_d$ to the plane is defined, the center of $G$ is the unique fixed point in $G$. In this case, the fixed point is of Type (1).\\

    
      \noindent Case II: Every fixed object of $F$ is subtended by leaves of the lamination. The middle lamination of Figure \ref{fig:Cases} is a $\sigma_4$ example: there is one fixed leaf and the degree three fixed sector $F$ contains a degree 3 rotational triangle.  The fixed leaf and the fixed point in the degree three fixed sector are subtended by pullbacks of the triangle.\\

    \noindent For each fixed object, we pick a sufficiently close leaf $l$ of the lamination $\mathcal{L}$ so that $l$ is repelled from the fixed object, but the image of $l$ remains in the fixed sector $F$. Let $\widetilde{\mathbb D}\subset F$ be the disk bounded by our chosen leaves, the intervals of $\mathbb{S}$, and the endpoints of adjacent leaves in $F$. We retract every point of $\mathbb{D} \setminus \widetilde{\mathbb{D}}$ to the nearest point of the boundary of $\widetilde{\mathbb D}$ by a retraction mapping $r$. $\widetilde{\mathbb D}$ does not contain any fixed point of $\sigma_d$ on $\mathbb{S}$. We have, $r(\sigma_d^{\#}(\widetilde{\mathbb D}))\subset \widetilde{\mathbb D}$. Since $r \circ \sigma_d^{\#}$ is continuous, Brouwer's fixed point theorem shows at least one fixed point exists in $\widetilde{\mathbb D}$. However, these fixed points can not occur on $\partial \widetilde{\mathbb D}$, so they must occur inside $F$ as $\widetilde{\mathbb D}\subset F$.
    
    In a hyperbolic lamination, fixed points are either the center of an invariant Fatou gap or an invariant equivalence class. In this case, we do not have an invariant Fatou gap. Therefore, the fixed point must occur in an invariant finite gap. These are rotational polygons.\\


    \noindent Case III: Some fixed objects of $F$ are subtended by leaves of the lamination, and some are not.  The rightmost lamination in  Figure \ref{fig:Cases} is a $\sigma_4$ example: there is one fixed leaf and in the degree three fixed sector, one critical chord touches a fixed point, but the other touches a period two leaf in $F$. \\
    
    \noindent As in Case I, the fixed objects not subtended by the leaves of the lamination are on the boundary of an invariant critical gap.  As we have defined $\sigma_d^{\#}$, the center of this invariant critical Fatou gap is fixed.



\end{proof}

\begin{ques}
    With reference to Theorem \ref{IFP}, we raise the question: Under what conditions is the FPP, $\mathcal{P}$, a sub-lamination of $\mathcal{L}$? 
\end{ques}

\section{Unicritical and Maximally Critical Correspondence}

\begin{defn}[Terminal and Non-Terminal Critical Sectors]
    We will say a critical sector is {\em terminal} if is formed by a singular critical chord of length $\dfrac{1}{d}$. Otherwise a critical sector is formed by two or more critical chords and will be called {\em non-terminal}.
\end{defn}

\begin{defn}[Unicritical Lamination] \label{unicritical}
A $d$-invariant lamination which is compatible with an all critical $d$-gon is called a {\em unicritical lamination}. We also refer to this as a {\em globally unicritical lamination}.
\end{defn}

\begin{defn}[Maximally Critical Lamination]
A $d$-invariant lamination which has a polygon with $d-1$ degree 2 Fatou gaps on sides of the polygon in separate orbits is called a {\em maximally critical lamination}. 
\end{defn}

\begin{defn}[Locally Unicritical Lamination] \label{locunicritical}
    A $d$-invariant lamination with a rotational polygon which is in a terminal critical sector of an all critical $k$-gon where $k<d$ is called a {\em locally unicritical lamination}.
\end{defn}

\begin{defn}[Locally Maximally Critical Lamination]
  A $d$-invariant lamination which has a rotational polygon with $k-1$ degree 2 Fatou gaps in separate orbits where $k<d$ is called a {\em locally maximally critical lamination}.  
\end{defn}

\begin{rem}
    We will refer to our local laminations as {\em canonical} if the critical portrait intersects the vertices of the polygon when possible. See Figure \ref{fig:fixed flowers} for examples of both canonical locally unicritical and locally maximally critical laminations.
\end{rem}

\begin{lem}
The length of the major of a unicritical polygon is within $\dfrac{1}{d(d+1)}$ of critical length $\left(\dfrac{1}{d}\right)$.
\end{lem}

\begin{proof}
    We know the length of any leaf in a lamination must get at least as long as $\dfrac{1}{d+1}$ (\cite{Cosper:2016}). Since the critical length for the critical chord associated with a unicritical major is $\dfrac{1}{d}$, our major leaf must be within $\dfrac{1}{d}-\dfrac{1}{d+1}=\dfrac{1}{d(d+1)}$ of critical length.
\end{proof}

The following definition and theorem are adapted from \cite{Schleicher:1999}.

\begin{defn}[Co-root]
A {\em co-root} is a point, other than an endpoint of the major, in the boundary of the central gap of the unicritical lamination that is fixed under the first return map.
\end{defn}

\begin{thm}[Co-root Theorem]\label{thm: co-root}
Given a degree $d$ unicritical lamination, there will be $d-2$ co-roots compatible with our unicritical polygon. The distance between co-roots is greater than $1/d$.
\end{thm}

The following algorithm gives rise to the existence of co-roots for all unicritical laminations. 

\begin{rem}[The Generalized Lavaur's Algorithm, Section 6 of \cite{Bhattacharya:2021}]\label{thm: lavaurs} 
For all unicritical laminations of any degree, we can uniquely find the corresponding minor leaf that generates the lamination. In this process of finding these minors, the algorithm \enquote{skips} over certain points of every period. These points that are skipped over are the desired co-roots. However, the Lavaur's Algorithm identifies the minor leaf rather than the major. Thus, the \enquote{co-roots} it identifies are actually the images of the points described in the co-root definition. These points are in the same orbit so we will use the term interchangeably when no confusion will arise. We will identity a co-root with the leaf with which it is associated  in its orbit. 
\end{rem}

\begin{thm}[Uniqueness of Rotation Numbers in a Critical Sector]\label{unique rotation number}
    Each critical sector has a unique rotation number. In other words, a non-zero rotation set and a fixed point cannot exist in the same critical sector. \cite{Jamie:2003}
\end{thm}

\begin{thm}\label{no adj rot}
    There cannot be rotational sets in two adjacent critical sectors.
\end{thm}

\begin{proof}
    Assume we have two adjacent critical sectors, $S_1$ and $S_2$, both with a rotational set within it. We have three cases: two terminal adjacent sectors, one terminal and one non-terminal adjacent sectors, and two non-terminal adjacent sectors.

    We observe that this result follows immediately in the easiest case that two critical sectors form a singular arc of the circle. The sectors are length $\dfrac{1}{d}$; the arc they form together is length $\dfrac{2}{d}$. We will have a fixed point in that arc as the fixed points are distributed $\dfrac{1}{d-1}<\dfrac{2}{d}$ apart contradicting Theorem \ref{unique rotation number}.

    In the mixed case, we have two sub-cases: one where the non-terminal critical sector has two arcs on the circle and the other where it has three or more. When the non-terminal sector has two arcs on the circle and is adjacent to a terminal critical sector, its longest critical chord will be of length $\dfrac{2}{d}$. Thus, by the previous argument, we will be able to fit a fixed point underneath this non-simple sector again contradicting Theorem \ref{unique rotation number}. Now we consider the case when the non-terminal sector has three or more arcs on the circle. Our two rotational sectors take up $\dfrac{2}{d}$ of the circle. In the remaining portion of the circle, we can fit up to $d-2$ fixed points in the other critical sectors. This leaves one fixed point left to place. Since our other critical sectors already have as many fixed points as allowed, this forces a fixed point in one of our rotational sectors contradicting Theorem \ref{unique rotation number}.

    The final case of two non-terminal adjacent sectors is the same as the previous sub-case.

\end{proof}

The following corollary follows immediately from Theorem \ref{no adj rot}.

\begin{cor}

A unicritical rotational polygon has all fixed points subtended by one of its sides.

\end{cor}

The following theorem follows from Definitions \ref{unicritical} and \ref{locunicritical}.

\begin{thm} \label{local to global}
    A polygon (or leaf) is compatible with a locally unicritical lamination if, and only if it is compatible with a globally unicritical lamination.
\end{thm}




The global correspondence between unicritical rotational laminations and maximally critical rotational laminations was shown in \cite{Burdette:2022}. The language used between that paper and this paper differ. The global correspondence is stated in the following two theorems:

\begin{thm}[SCM to MAC Lamination]\cite{Burdette:2022} \label{global correspondence1}
Let $\mathcal{S}$ be a canonical SCM lamination containing a rotational SCM polygon, $P$. Then, there exists a unique corresponding MAC lamination, $\mathcal{L}$, with a leaf $M$ as its MAC leaf. 
\end{thm}

\begin{thm}[MAC to SCM Lamination]\cite{Burdette:2022} \label{global correspondence2}
\label{thm: MAC to SCM Lamination}
Let $\mathcal{L}(M)$ be a MAC lamination with the MAC leaf $M$. There is a canonical SCM lamination $\mathcal{S}(M)$ that contains an SCM rotational $k(d-1)$-gon polygon $P(M)$. Here $k$ is the period of when the rotational polygon returns to itself.

\end{thm}

In the language of this paper, an SCM lamination is a maximally critical lamination, and a MAC lamination is a (globally) unicritical lamination.


\begin{figure}[H]
    \centering
    \includegraphics[width=2in]{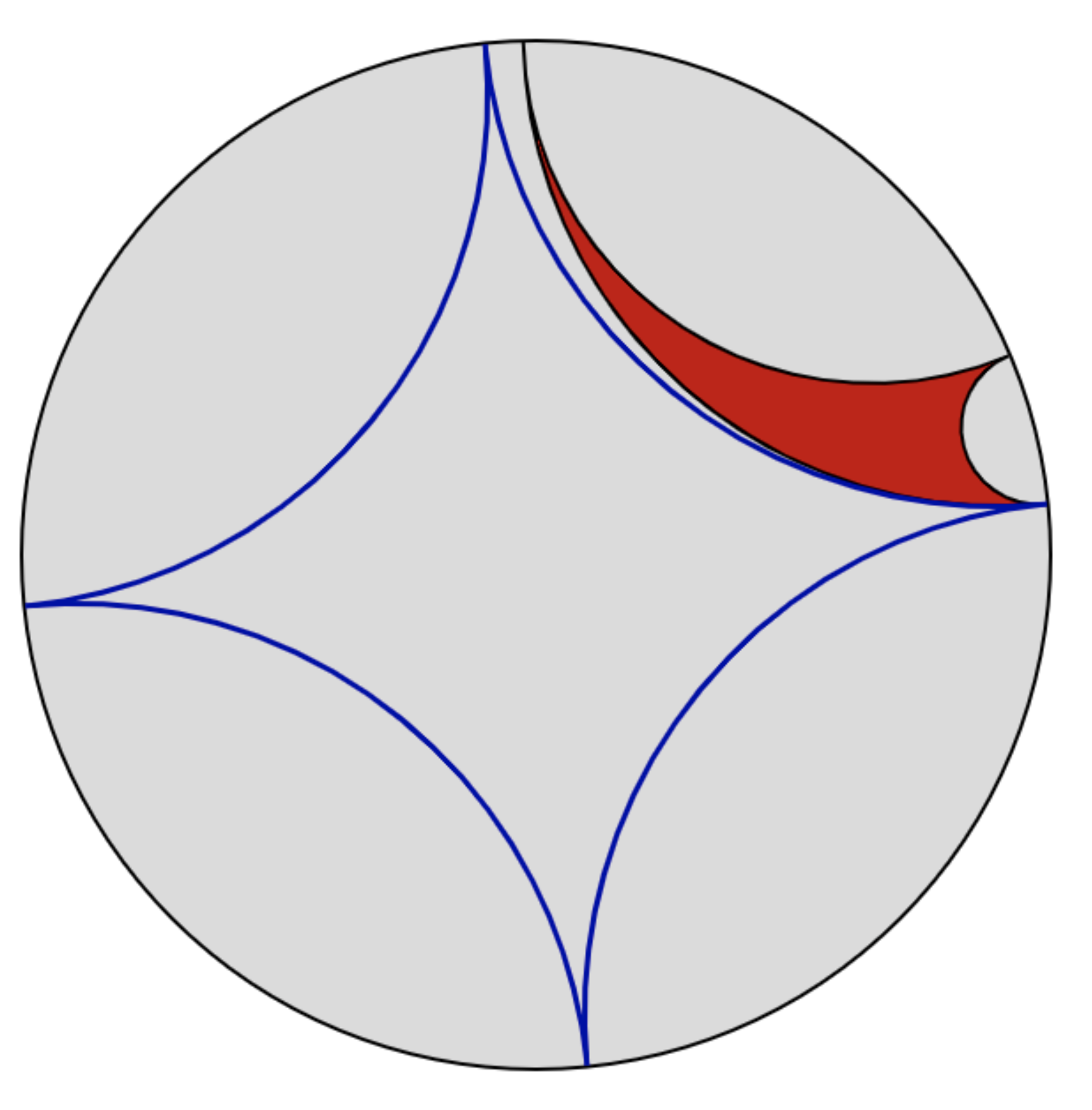}
    \includegraphics[width=2in]{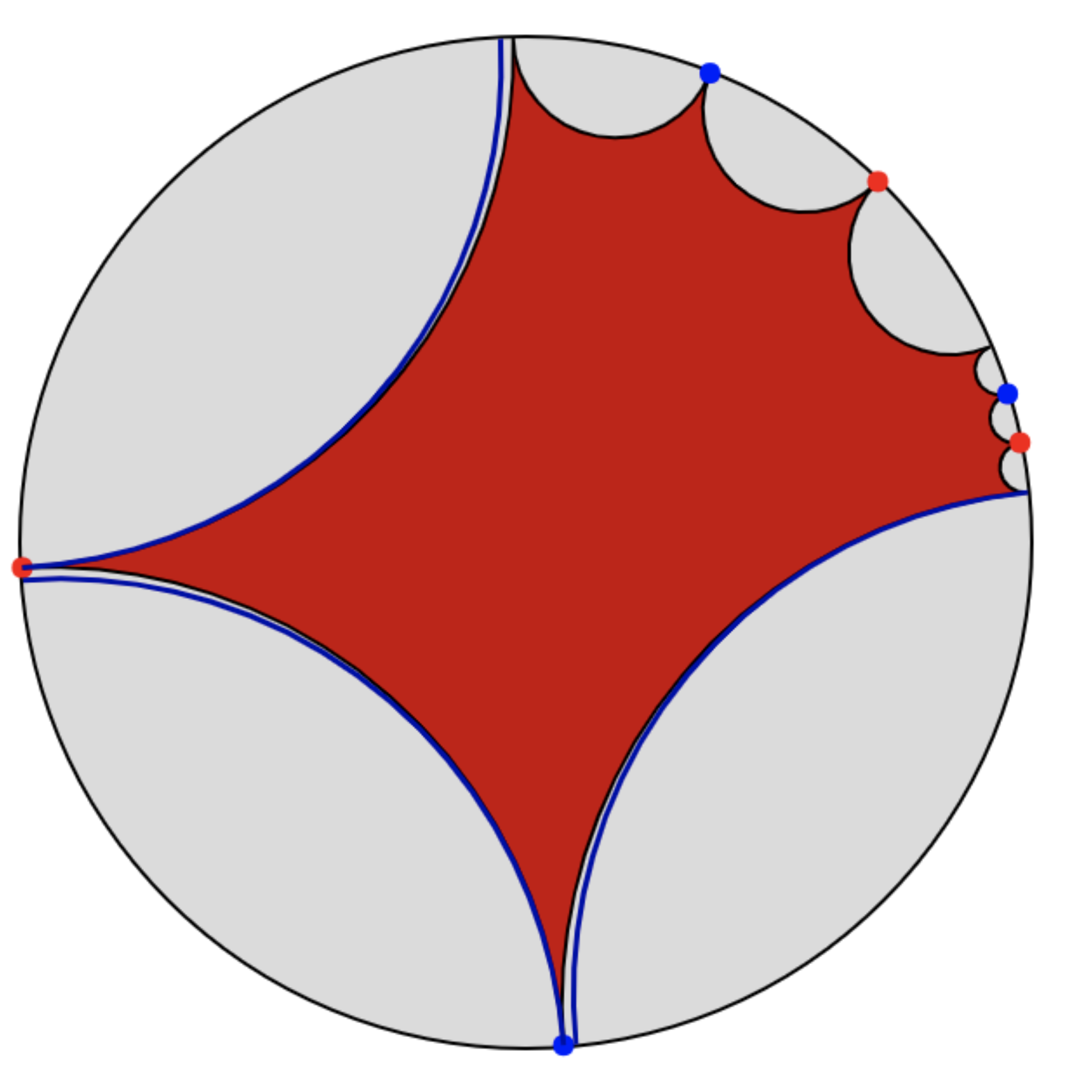}
    \caption{A globally unicritical rotational polygon (left) and a globally maximally critical rotational polygon (right).}
    \label{fig:global correspondence}
\end{figure}

Now, we wish to establish the local correspondence. Comparing Figures \ref{fig:global correspondence} and \ref{fig:fixed flowers}, we can visually see what changes when we are in the local case. 

In the global case, the Generalized Lavaur's algorithm \cite{Bhattacharya:2021} gives $d-2$ co-roots for globally unicritical rotational polygons. These co-roots and the original vertices of the rotational polygon give the vertices for the maximally critical rotational polygon. See Figure \ref{fig:global correspondence}.

In the local case, we can again use the Generalized Lavaur's algorithm to find $d-2$ co-roots. However, not all will be compatible with our locally unicritical polygon. Here, there will be $d^\prime-2$ (local degree) compatible co-roots for the locally unicritical rotational polygon. These compatible co-roots for the local case can be used to find an associated locally maximally critical polygon. See Figure \ref{fig:fixed flowers}.

\begin{thm}\label{find local coroots}
Let $\lam$ be a canonical lamination with a locally unicritical rotational polygon. Then, there exists a canonical lamination with a locally maximally unicritical $k(d^\prime-1)$-gon.

\end{thm}

\begin{proof}
    This follows from Theorems \ref{local to global} and \ref{global correspondence2}.
\end{proof}

\begin{rem}
    As in the global correspondence, the locally maximally critical polygon will have adjacent major leaves.
\end{rem}

\begin{thm}
Given a canonical lamination, $\lam$, with a locally maximally critical polygon with adjacent majors, there exists a canonical lamination with a unique locally unicritical rotational polygon.

\end{thm}

\begin{proof}
    This follows from Theorems \ref{local to global} and \ref{global correspondence1}.
\end{proof}

\begin{figure}[H]
    \centering
    \includegraphics[width = 2in]{images/s4_loc_uni_rabbit.png}
     \hspace{.1 in}
    \includegraphics[width = 2in]{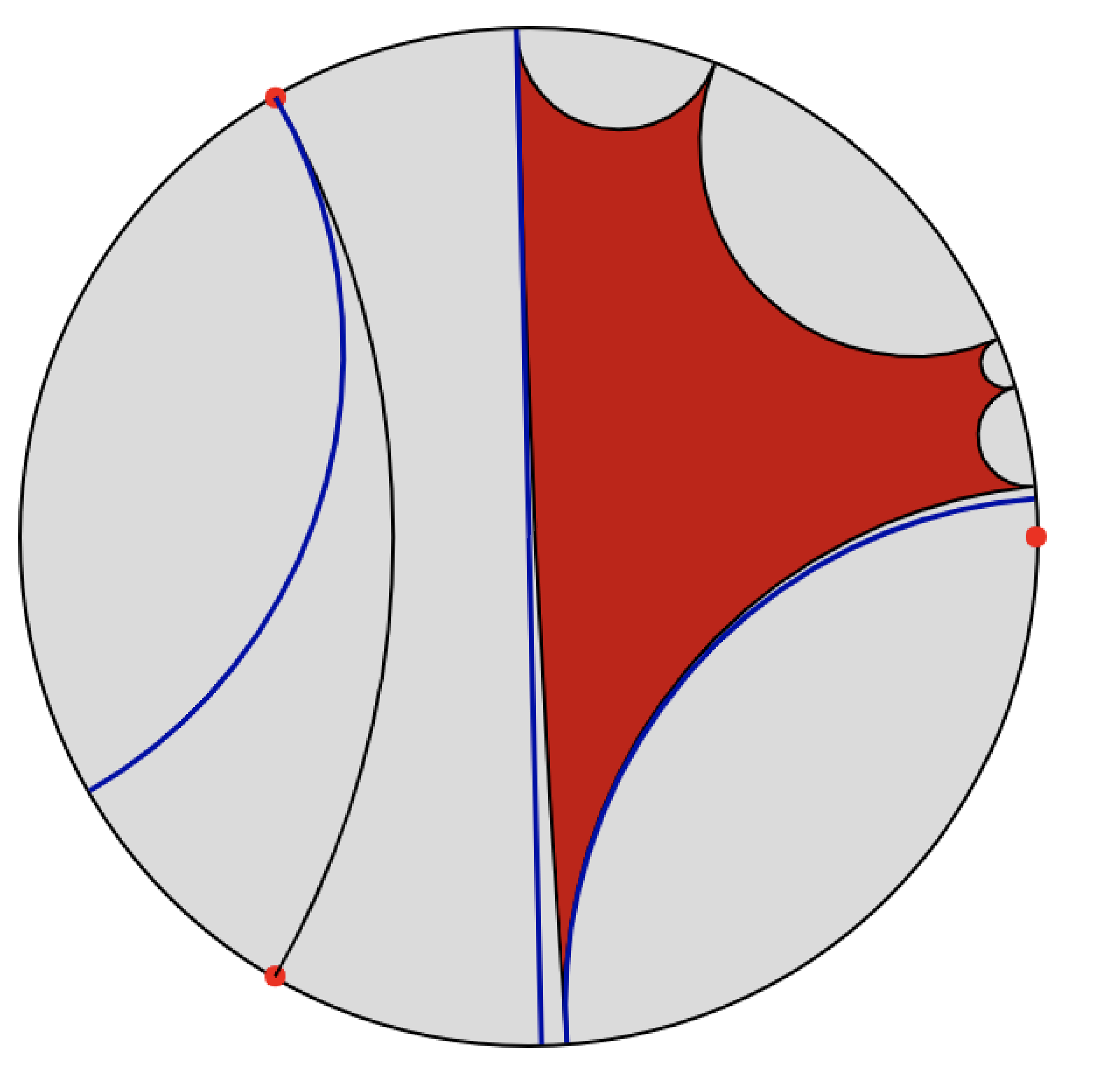}

    \caption[Simplest Fixed Point Portrait]{Here we have our locally unicritical rotational polygon from Figure \ref{loc uni full diagram}} (left) and the corresponding locally maximally critical polygon (right).
    \label{fig:fixed flowers}
\end{figure}

\begin{rem}
    The global count of fixed point portraits in Subsection \ref{globalcount} can be extended to a local count given a locally unicritical rotational polygon whose canonical all critical polygon separates the fixed points.  Details are left to the reader.
\end{rem}

\newpage

\bibliographystyle{amsplain}

\end{document}